\documentclass[a4paper, 11pt]{article}

\usepackage{mathrsfs,amssymb,amsmath, amsthm,color,tikz,epic}

\usepackage{multirow}

\usepackage[top=1.31in,bottom=1.29in,left=1.1in,right=0.9in]{geometry}
\usepackage[numbers]{natbib}
\usepackage[all]{xy}

\setcitestyle{open={},close={}}

\numberwithin{equation}{section}\theoremstyle{definition}
\swapnumbers

 \newtheorem{Theorem}[equation]{Theorem}
 \newtheorem{Prop}[equation]{Proposition}
 \newtheorem{Lemma}[equation]{Lemma}
 \newtheorem{Cor}[equation]{Corollary}
 \newtheorem{Notation}[equation]{Notation}
 \newtheorem{Defn}[equation]{Definition}
 \newtheorem{Example}[equation]{Example}
 \newtheorem{Remark}[equation]{Remark}

\newtheorem{Setup}[equation]{Set up}
\newtheorem{Facts}[equation]{Facts}
\newtheorem{Procedure}[equation]{Procedure}

\newtheorem{Summary}[equation]{Summary}

 \makeatletter
\def\enumerate{\begingroup\ifnum\@enumdepth>3\@toodeep\else
      \advance\@enumdepth\@ne
      \edef\@enumctr{enum\romannumeral\the\@enumdepth}%
      \topsep\z@\parskip\z@
      \list{\csname label\@enumctr\endcsname}
        {\@nmbrlisttrue\let\@listctr\@enumctr
         \parsep\z@\itemsep\z@\topsep\z@
         \setcounter{\@enumctr}{0}
         \def \fMakelabel##1{\hss\llap{\rm ##1}}
       }\fi}

 \makeatother

\def\fh{\mathfrak h}

\def\p{\mathfrak p}

\def\n{\mathfrak n}

\def\f{\mathfrak f}

\def\v{\mathfrak v}
\def\w{\mathfrak w}
\def\y{\mathfrak y}
\def\z{\mathfrak z}
\def\fc{\mathfrak c}
\def\fr{\mathfrak r}
\def\fa{\mathfrak a}
\def\fe{\mathfrak e}
\def\fri{\mathfrak i}

\def\bfn{\text{\bf n}}
\def\fI{\mathfrak I}

\def\fE{\mathfrak E}

\def\fS{\mathfrak S}

\def\fT{\mathfrak T}
\def\fTn{\fT^{\n}}

\def\C{\mathbb C}
\def\N{\mathbb N}

\def\F{\mathbb F}

\def\cO{\mathcal O}
\def\cH{\mathcal H}
\def\cB{\mathcal B}
\def\cT{\mathcal T}

\def\cR{\mathcal R}

\def\cI{\mathcal I}

\def\cN{\mathcal N}
\def\cC{\mathcal C}
\def\cZ{\mathcal Z}

\def\cS{\mathcal S}

\def\cK{\mathcal K}
\def\cA{\mathcal A}
\def\cW{\mathcal W}

\def\bfy{\text{\bf y}}
\def\bfw{\text{\bf w}}

\def\bbeta{\pmb{\Vec{\beta}}}
\def\xxi{\pmb{\Vec{\xi}}}

\def\bB{\cB^\circ}
\def\bC{\cC^\circ}

\def\hC{\hat\cC}
\def\hH{\hat\cH}
\def\hB{\hat\cB}

\def\bN{\cN^\circ}

\def\G{\Gamma}
\def\hG{\hat\Gamma}

\def\hL{\hat L}

\def\Pl{\mathbb{PL}}
\def\rPl{\mathbb{PL}_r}
\def\lPl{\mathbb{PL}_\ell}

\def\cObi{\cO^{\bi}}
\def\cOr{\cO^r}
\def\cOl{\cO^l}
\def\cOrl{\cO^{r\cap l}}

\def\V{\hat{V}}

\DeclareMathOperator{\Hom}{Hom}

\DeclareMathOperator{\col}{col}

\DeclareMathOperator{\tr}{tr}

\DeclareMathOperator{\End}{End}

\DeclareMathOperator{\rank}{rank}

\DeclareMathOperator{\main}{main}

\DeclareMathOperator{\nker}{nker}

\DeclareMathOperator{\suppl}{suppl}

\DeclareMathOperator{\bi}{bi}
\DeclareMathOperator{\im}{im}
\DeclareMathOperator{\supp}{supp}

\DeclareMathOperator{\Lie}{Lie}
\begin{document}


\title{Coadjoint orbits of Row  Closed Subgroups of $U_n(q)$}

\author{Qiong Guo$^{*}$, Richard Dipper$^{**}$\\ \\$^{*}${\footnotesize College of Sciences, Shanghai Institute of Technology} \\ {\footnotesize 201418 Shanghai, PR China}
\\ \scriptsize{E-mail:qiongguo@hotmail.com}\smallskip \\$^{**}$ {\footnotesize Institut f\"{u}r Algebra und Zahlentheorie}\\ {\footnotesize Universit\"{a}t Stuttgart, 70569 Stuttgart, Germany}
\\ \scriptsize{E-mail: richard.dipper@mathematik.uni-stuttgart.de}\\
\setcounter{footnote}{-1}\footnote{
{\scriptsize This work was  supported by NSFC no.\ 11601338}}
\setcounter{footnote}{-1}
\setcounter{footnote}{-1}}
\date{}


 \maketitle

\begin{abstract}
Let $n$ be a natural number, let $q$ be a prime power, and let $U_n(q)$ denote the group of unitriangular $n\times n$ matrices over the finite field $\F_q$ with $q$ elements. Row closed and column closed subgroups $U$ of $U_n(q)$ are special pattern subgroups obtained by deleting entire rows or, respectively, entire columns (apart from the diagonal entries). The supercharacters of the Andr\'{e}-Yan supercharacter theory of $U$ are afforded by the orbit modules arising from a monomial action of $U$ on the character group of the Lie algebra of $U$. We classify the corresponding orbits for row closed subgroups $U$, that is, we determine all orbits and identify which of the associated orbit modules are isomorphic and which afford orthogonal supercharacters. The classification for column closed subgroups follows from the mirror map, which reflects matrices across the antidiagonal.
\end{abstract}

\section{Introduction}  Determining the conjugacy classes of the finite groups $U_n(q)$ of upper unitriangular $n\times n$-matrices simultaneously for all prime powers $q$ and $n\in\N$ is known to be a wild problem in the categorical sense, see e.g. [\cite{vera1}], [\cite{vera}]. Thus, classifying the conjugacy classes and irreducible characters for $U_n(q)$ seems to be out of reach.  

 Andr\'{e} in [\cite{andre}], and later  Yan  in [\cite{yan thesis}, \cite{yan}], constructed characters of $U_n(q)$, called basic and transition characters, respectively, using a cruder version of Kirillov's orbit method [\cite{kirillov}]. These characters  provide an approximation to the classification problem of the irreducible characters of $U_n(q)$. Their approach was axiomatized by  Diaconis and  Isaacs in  [\cite{super}] through the notion of a supercharacter theory of a finite group $G$: Such a theory consists of a partition of $G$ into pairwise disjoint unions of conjugacy classes, called superclasses, and a set of pairwise orthogonal complex characters, called supercharacters, such that every irreducible complex character of $G$ occurs as a constituent of precisely one supercharacter. Moreover, the superclasses and supercharacters are in one-to-one correspondence, and the supercharacters are constant on the superclasses.

Diaconis and Isaacs extended the Andr\'{e}--Yan supercharacter theory to so-called algebra groups. Important examples of algebra groups are pattern subgroups of $U_n(q)$. These are defined by requiring the matrix entries in certain prescribed positions to vanish, subject to appropriate closure conditions; see \ref{startfacts} below. These conditions are satisfied, in particular, when the set of vanishing positions consists of all entries above the diagonal in a prescribed collection of rows or columns. The corresponding pattern subgroups $U$ of $U_n(q)$ are the row closed and column closed subgroups appearing in the title of this paper. The investigation of their Andr\'{e}--Yan supercharacters is the main theme of the present work.

One approach to constructing the Andr\'{e}--Yan supercharacters of pattern subgroups $U$ of $U_n(q)$ utilizes the natural left and right actions of $U$ on its Lie algebra $V = \Lie(U) = \{u-1\,|\,u\in U\}$. These actions induce left and right actions on the dual space $V^*$ and, in turn, a monomial biaction of $U$ on the set $\V$ of linear complex characters of the abelian group  
$(V,+)$. The resulting orbits may be viewed as finite analogues of Kirillov's
coadjoint orbits and are therefore also referred to as coadjoint
orbits.
The linearisation $\C V$ is isomorphic to the regular $\C U$-bimodule. Its one-sided orbit submodules afford the Andr\'{e}-Yan supercharacters of $U$. Consequently,  the characters afforded by any two such orbit modules are either equal or orthogonal. 

The aim of this paper is to apply this method to row closed  subgroups $U$ of $U_n(q)$ and to classify their coadjoint orbits under the monomial action on $\V$. This involves two problems. First, one must list the orbits, for example by finding distinguished elements of $\V$ that parametrize them. Second, one must determine which of the associated complex orbit modules are isomorphic and which afford orthogonal characters, that is, which have no irreducible constituents in common.

For column closed subgroups $U$, the classification of the right orbits of $U$ on $\V$ was carried out in [\cite{DG3}]. To extend these results to left orbits and also to row closed pattern subgroups, we use the mirror map obtained by reflecting $n\times n$ matrices across the antidiagonal. This map interchanges row closed and column closed patterns, as well as left and right orbits. Consequently, the classification of the left, respectively right, orbits of $U$ is equivalent to the classification of the right, respectively left, orbits of the reflected group.

We therefore concentrate on classifying the coadjoint orbits of row closed subgroups $U$ acting on $\V$. Applying the mirror map then immediately yields the classification of the right orbits and their orbit modules in the column closed case, thereby providing a new proof of the corresponding results of [\cite{DG3}]. The mirror map also yields classifications of the left orbits and biorbits in the column closed case, which were not previously known. Since this final step is straightforward, its details are left to the reader.

\section{Preliminaries}

We begin by fixing some basic notation for vectors and matrices over $\F_q$:

\begin{Notation}\label{colvects} 
Let $\cR$ and $\cC$ be finite ordered sets.
\begin{itemize}
\item[(1)] The $\F_q$-space of maps from $\cR$ to $\F_q$ (viewed as ordered tuples
$\bbeta=(\beta_r)=(\beta_r)_{r\in\cR}$ of elements $\beta_r\in\F_q$ indexed by $\cR$)
is denoted by $\F_q^{\cR}$. Similarly, $M_{\cR\times\cC}(q)$ denotes the $\F_q$-space of matrices whose entries are indexed by $\cR\times\cC$. We set
$M_\cR(q)=M_{\cR\times\cR}(q)$. For $l,k\in\N$, $\cR=\{1,\ldots,l\}$, and
$\cC=\{1,\ldots,k\}$, we replace the indices $\cR$ and $\cC$ simply by $l$ and $k$, respectively.

\item[(2)] Let $B \in M_{\cR\times\cC}(q)$. Then $\rank(B)$ denotes the rank of $B$. Let $i\in\cR$ and $j\in\cC$. We denote the $i$th row vector of $B$ by $\fr_i(B)\in\F_q^{\cC}$ and the $j$th column vector by $\fc_j(B)\in\F_q^{\cR}$. The entry of $B$ in position $(i,j)$ is denoted by $B_{ij}$, and the $(i,j)$th matrix unit of $M_{\cR\times\cC}(q)$ is denoted by $e_{ij}$. Thus
\[
\{e_{ij}\mid(i,j)\in\cR\times\cC\}
\]
is the natural basis of $M_{\cR\times\cC}(q)$. The dual basis of the dual space
\[
M_{\cR\times\cC}(q)^*
=
\Hom_{\F_q}(M_{\cR\times\cC}(q),\F_q)
\]
is given by the coordinate functions
\[
\varepsilon_{ij}=e_{ij}^*:B\longmapsto B_{ij},
\]
where $B\in M_{\cR\times\cC}(q)$ and $(i,j)\in\cR\times\cC$. For $A\in M_{\cR\times\cC}(q)$ we define 
\[
A^*=\sum_{(i,j)\in\cR\times\cC}A_{ij}\varepsilon_{ij}\in\Hom_{\F_q}(M_{\cR\times\cC}(q),\F_q).
\]
Thus 
\[ A^*:M_{\cR\times\cC}(q)\longrightarrow \F_q:B\longmapsto\sum_{(i,j)\in\cR\times\cC}A_{ij}B_{ij} = \tr(A^tB),
\]
where $\tr$ denotes the trace map on square matrices.

The transpose of matrix $B\in M_{\cR\times\cC}(q)$ is denoted by
\[
B^t\in M_{\cC\times\cR}(q),
\]
and if $\cR = \cC$ and $B$ is invertible, we abbreviate 
\[
B^{-t} = (B^{-1})^t = (B^t)^{-1}.
\]
For a row (respectively column) vector $\v\in\F_q^\cC$ (respectively $\v\in\F_q^\cR$), the corresponding column (respectively row) vector is denoted by $\v^t$. For $\v\in\F_q^\cR$ and $i\in\cR$, we denote the $i$th component of $\v$ by $\v_i$.

\item[(3)] We define the \textbf{support} of $B\in M_{\cR\times\cC}(q)$ by
\[
\supp(B)=\{(i,j)\mid i\in\cR,\; j\in\cC,\; B_{ij}\neq0\}.
\]
Thus,
\[
B=\sum_{{i\in\cR},\, {j\in\cC}}B_{ij}e_{ij}
 =\sum_{(i,j)\in\supp(B)}B_{ij}e_{ij}.
\]

\item[(4)] If $\cR'\subseteq\cR$, $\cC'\subseteq\cC$, $B\in M_{\cR\times\cC}(q)$, and
$Y\in M_{\cR'\times\cC'}(q)$, then
$B_{\cR'\times\cC'}$ denotes the restriction of $B$ to
$\cR'\times\cC'$ (viewing $B$ as a map from $\cR\times\cC$ to
$\F_q$). Moreover,
$Y^{\cR\times\cC}$ denotes the extension of $Y$ to
$\cR\times\cC$ by zeros. Thus
\[
B_{\cR'\times\cC'}
=
\sum_{(i,j)\in\cR'\times\cC'}B_{ij}e_{ij},
\qquad
Y^{\cR\times\cC}
=
\sum_{(i,j)\in\cR'\times\cC'}Y_{ij}e_{ij}
+
\sum_{(i,j)\in(\cR\times\cC)\setminus(\cR'\times\cC')}0\cdot e_{ij}.
\]
If no ambiguity arises, we henceforth omit the superscript
$\cR\times\cC$ and simply regard $Y$ as a matrix in
$M_{\cR\times\cC}(q)$ by setting
$Y_{ij}=0$ whenever
$(i,j)\in(\cR\times\cC)\setminus(\cR'\times\cC')$.
\item[(5)] For $\v\in\F_q^{\cR}$ and $\w\in\F_q^{\cR'}$, we use the analogous notation
$\v_{\cR'}\in\F_q^{\cR'}$ and $\w=\w^{\cR}\in\F_q^{\cR}$.

\hfill$\square$
\end{itemize}
\end{Notation}

Let $n$ be a natural number and let $\cR=\cC=\{1,\ldots,n\}$. Define
\[
\Phi^+=\{(i,j)\mid 1\le i<j\le n\},
\]
to be the set of \textbf{positive roots}. Then the group of upper unitriangular matrices in $M_n(q)$ is
\[
U_n(q)=\{A\in M_n(q)\mid \supp(A-E)\subseteq\Phi^+\},
\]
where $E=E_n$ denotes the $n\times n$ identity matrix
$
E=\sum_{1\le i\le n}e_{ii}.
$

We now give a brief account of the construction of supercharacters for pattern subgroups of the unitriangular group $U_n(q)$. For proofs and further details, we refer to [\cite{DG3}].

For any subset $J\subseteq\Phi^+$, we define
\[
V_J=\{A\in M_n(q)\mid\supp(A)\subseteq J\}
=\bigoplus_{(i,j)\in J}\F_qe_{ij}.
\]
The natural projection of $M_n(q)$ onto $V_J$ is denoted by $\pi_J$. Thus, for $A\in M_n(q)$,
\[
\pi_J(A)=\sum_{(i,j)\in J}A_{ij}e_{ij}.
\]

For $1\leq i,j\leq n$, $i\neq j$, and $\alpha\in\F_q$, define
\[
x_{ij}(\alpha)=E+\alpha e_{ij}\in M_n(q)
\]
and
\[
X_{ij}=\{x_{ij}(\alpha)\mid\alpha\in\F_q\}.
\]
Then $X_{ij}$ is a subgroup of the general linear group $GL_n(q)$, isomorphic to the additive group $(\F_q,+)$ of the field $\F_q$. It is called the \textbf{root subgroup} corresponding to the root $(i,j)\in\Phi^+$. Here we identify $\Phi^+$ with the set of positive roots of the root system $\Phi$ associated with $GL_n(q)$ and its Borel subgroup of upper triangular invertible matrices.

\begin{Defn}
A subset $J$ of $\Phi^+$ is said to be \textbf{closed} if $(i,j),(j,k)\in J$ implies $(i,k)\in J$.
\hfill$\square$
\end{Defn}

\begin{Facts}\label{startfacts}
The following facts are well known (see, e.g., [\cite{carter}]).
Let $J\subseteq\Phi^+$. Then
\[
U_J=\{A\in U_n(q)\mid\supp(A-E)\subseteq J\}
\]
is a subgroup of $U_n(q)$ if and only if $J$ is closed. In this case, we call $U_J$ a \textbf{pattern subgroup} of $U_n(q)$. Moreover, $U_J$ is generated by the root subgroups $X_{ij}$ with $(i,j)\in J$. Indeed, after fixing an arbitrary linear ordering of $J$, every element $u$ of $U_J$ can be written uniquely in the form
\begin{equation}\label{uniquelywritten}
u=\prod_{(i,j)\in J}x_{ij}(\alpha_{ij}).
\end{equation}
Here the elements $\alpha_{ij}\in\F_q$, for $(i,j)\in J$, are uniquely determined by $u$, provided that the product in \ref{uniquelywritten} is taken in the fixed ordering of $J$. In particular, $|U_J|=q^{|J|}$.
\hfill$\square$
\end{Facts}

Fix a closed subset $J$ of $\Phi^+$. We set
\[
V_J=\{A\in M_n(q)\mid\supp(A)\subseteq J\}
     =\{u-E\mid u\in U_J\}.
\]
We remark in passing that $V_J$ is the Lie algebra of $U_J$, although we shall not need this fact in the sequel. Note that $U_J$ acts on $V_J$ by matrix multiplication from the left and from the right. Hence $U_J$ acts on the dual space $V_J^*$ by
\begin{equation}
(u_1.\ell.u_2)(A)=\ell(u_1^{-1}Au_2^{-1}),
\qquad
\text{for }A\in V_J,\ \ell\in V_J^*,\ u_1,u_2\in U_J.
\end{equation}

For $A\in V_J$, we denote by $A^*$ the restriction  of
\[
\sum_{(i,j)\in J}A_{ij}\varepsilon_{ij}\in M_n(q)^*
\]
to $V_J$, abusing notation, if no confusion can arise. In [\cite{DG3}, 3.1] we proved the following.

\begin{Lemma}\label{dotstar}
Let $u_1,u_2\in U_J$ and $A\in V_J$. Then
\[
u_1.A^*.u_2
=
(u_1^{-t}Au_2^{-t})^*
=
\left(\pi_J(u_1^{-t}Au_2^{-t})\right)^*.
\]
This defines a bi-permutation action of $U_J$ on $V_J^*$.
\hfill$\square$
\end{Lemma}

We remark that the matrix product $u_1^{-t}Au_2^{-t}$ in Lemma~\ref{dotstar} does not, in general, belong to $V_J$.

Next, we fix a nontrivial linear character
$
\theta:(\F_q,+)\to\C^* 
$
of the additive group of the field $\F_q$ into the multiplicative group $\C^* = \C\setminus\{0\}$ of the complex numbers. For any 
$\alpha\in\F_q$, we define
\begin{equation}\label{linchars}
\theta_\alpha:(\F_q,+)\longrightarrow\C^*:
\beta\longmapsto\theta(\alpha\beta)\in \C^*\text{ for }\beta\in\F_q.
\end{equation}
Then
$
\{\theta_\alpha\mid\alpha\in\F_q\}
$
is precisely the set of linear characters of the additive group $(\F_q,+)$.

For any $\ell\in V_J^*$, the composite map
$
\theta\circ\ell:V_J\to\C^*
$
is a linear character of the additive group $(V_J,+)$. Indeed, the map
$
\ell\longmapsto\theta\circ\ell
$
is a bijection from $V_J^*$ onto the group $\V_J$ of linear characters of the additive group $(V_J,+)$.

For $A\in V_J$, we henceforth denote the linear character $\theta\circ A^*$ by $[A]$. In the following, we speak of entries, columns, rows, etc. of matrices $A\in V_J$ and of characters $[A]\in\V_J$ alike.

In [\cite{DG3}], the linear character $[A]$ is also called a \textbf{lidempotent}, since, up to a power of $q$, it is an idempotent of the additive group
$V_J=\Lie(U_J)$ if we identify the group algebra $\C V_J$ with the space of maps from $V_J$ to $\C$.

For $u\in U_J$ and $A\in V_J$ we define $u.[A]=\theta\circ (u.A^*)$ and $[A].u=\theta\circ (A^*.u)$. Then this defines a permutation biaction of $U_J$ on $\V_J$, which may be turned by [\cite{DG3},3.14] into a monomial action setting 
\begin{eqnarray} \label{MonomialBiaction}
u[A]&=& \theta\big(A^*(u-E)\big)u.[A] \\
{[A]}u&=& \theta\big(A^*(u-E)\big)[A].u
\end{eqnarray}

\begin{Remark}\label{PermAction}
For most of this paper we investigate the coadjoint orbits and therefore deal only with the permutation action of $U_J$ on $\V_J$. Throughout, we use the notation introduced above, denoting the permutation action by a dot and the monomial action by juxtaposition.
Note that, for the study of coadjoint orbits $\cO\subseteq\V_J$, it suffices to consider the permutation action. Only when we consider the orbit modules $\C\cO$ do the monomial coefficients become relevant. Thus, throughout this paper, the notation $\C\cO$ always refers to the corresponding monomial $\C U_J$-module.
\end{Remark}

The following theorem is a special case of a more general result of Diaconis and Isaacs on $\F_q$-algebra groups from their fundamental paper [\cite{super}]. For details, we refer to [\cite{DG3}, 2.7].

\begin{Theorem}\label{supercharacters}
Write $U=U_J$. The left and right actions of $U$ on $\C\V_J$, with monomial basis $\V_J$ defined by (\ref{MonomialBiaction}), turn $\C\V_J$ into a $U$-bimodule. Moreover, the map
\[
f:\C\V_J\longrightarrow\C U:
[A]\longmapsto
\sum_{u\in U}\theta(A^*(u-E))u
=
\sum_{u\in U}
\left(
\prod_{(i,j)\in J}\theta(A_{ij}u_{ij})
\right)u
\]
is an isomorphism of $U$-bimodules onto the regular bimodule
$
{}_{\C U}\C U_{\C U}.
$
\hfill$\square$
\end{Theorem}

For root subgroup elements, the permutation action on $\V_J$ underlying the monomial action has a particularly simple description:

\begin{Cor}\label{restricted action}[\cite{DG3},3.14]
 For $(i,j)\in J\subseteq \Phi^+$ and $A\in V_J$, we have setting $\beta = A_{ij}$
\begin{itemize}
\item [(1)] $[A] x_{ij}(\alpha) =\theta(-\alpha A_{ij})[B] = \theta_{-\beta}(\alpha)[B]$, for $\alpha\in \F_q^*$. 
Here $B$ is obtained from $A$ by adding $-\alpha$ times column $j$ to column $i$ and then resetting all entries outside $J$ to zero. We call this ``{\bf restricted column operation}'' from right to left.
\item [(2)]  $ x_{ij}(\alpha) [A]=\theta (-\alpha A_{ij})[B]= \theta_{-\beta}(\alpha)[B]$,  for $\alpha\in \F_q^*$. 
Here $B$ is obtained from $A$ by adding $-\alpha$ times row $i$ to row $j$ and tand then setting all entries outside $J$ equal to zero. We call this ``{\bf restricted row operation}'' from top down. \hfill$\square$
\end{itemize}
\end{Cor}

In the following, we represent linear characters $[A]\in\V_J$ graphically by upper triangular matrices, omitting the entries below and on the diagonal. The remaining entries are to be interpreted as the coefficients of the coordinate functions 
$\varepsilon_{ij}$ followed by $\theta$. Moreover, the indices along the diagonal indicate both the row coordinates (to the right) and the column coordinates (upwards). Indices at the top of the triangle always indicate column coordinates. The permutation action of $x_{ij}(\alpha)$, for $(i,j)\in J$ and $\alpha\in\F_q$, on $[A]\in\V_J$ is illustrated as follows:

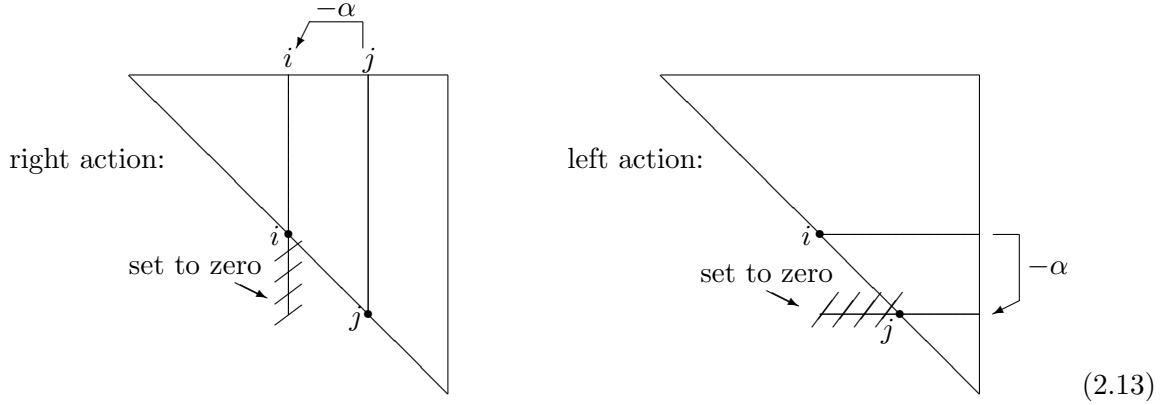
\begin{equation}\label{permaction}
\begin{picture}(320,140)


\put(-45,85){\makebox{right action:}}

\put(0,120){\line(1,0){120}}
\put(120,0){\line(0,1){120}}
\put(0,120){\line(1,-1){120}}

\put(90,30){\circle*{3}} 
\put(83,26){\makebox{$j$}}
\put(88,123){\makebox{$j$}}
\dottedline{.2}(90,120)(90,30)

\put(53,56){\makebox{$i$}}
\put(60,60){\circle*{3}} 
\put(58,123){\makebox{$i$}}
\dottedline{.2}(60,30)(60,120)

\multiput(55,50)(0,-8){4}{\line(4,3){10}}

\put(40,42){\vector(2,-1){12}}

\put(88,130){\line(0,1){10}}
\put(88,140){\line(-1,0){20}}
\put(68,140){\vector(-1,-2){5}}
\put(70,142){\makebox{$-\alpha$}}

\put(0,45){\makebox{set to zero}}


\put(165,85){\makebox{left action:}}

\put(200,120){\line(1,0){120}}
\put(320,0){\line(0,1){120}}
\put(200,120){\line(1,-1){120}}

\put(290,30){\circle*{3}}
\put(283,21){\makebox{$j$}}

\dottedline{.2}(260,30)(320,30)

\put(253,56){\makebox{$i$}}
\put(260,60){\circle*{3}} 
\dottedline{.2}(260,60)(320,60)

\put(325,60){\line(1,0){10}}
\put(335,60){\line(0,-1){25}}
\put(335,35){\vector(-2,-1){10}}
\put(338,45){\makebox{$-\alpha$}}

\put(215,42){\makebox{set to zero}}
\multiput(257,25)(8,0){4}{\line(3,4){10}}
\put(240,39){\vector(2,-1){12}}

\end{picture}
\end{equation}

\bigskip

For $A\in V_J$, we denote by $\cO_A^l$ and $\cO_A^r$ the left and right orbits of $\V_J$ generated by $[A]$, respectively, and by $\cO_A^{\bi}$ the corresponding biorbit. The characters afforded by these orbit modules are precisely the supercharacters of $U_J$ introduced by Diaconis and Isaacs in [\cite{super}], thereby generalizing the Andr\'e--Yan supercharacters for the special case $J=\Phi^+$ constructed in [\cite{andre}] and [\cite{yan}].

\begin{Remark}\label{basicLA}
Let $K$ be a field, $W$ a $K$-vector space with basis $Z$, and let $X,Y\subseteq Z$. Then basic linear algebra immediately yields
$
KX\cap KY = K(X\cap Y).
$
\hfill$\square$
\end{Remark}

\begin{Remark}{[\cite{DG3}, 3.16]} \label{HomSpace}
Recall the definition of the $\,\C U_J$-bimodule isomorphism
$f:\C\V_J\rightarrow\C U_J$ from Theorem~\ref{supercharacters}.
Observe that $f(\C\cO_A^l)=\C U_Jf([A])$ is a left ideal, $f(\C\cO_A^r)=f([A])\C U_J$ a right ideal, and 
$f(\C\cO_A^{\bi})=\C U_Jf([A])\C U_J$ a two-sided ideal of $\C U_J$. For $A,B\in V_J$, we have
\begin{eqnarray*}
\Hom_{\C U_J}\big(\C\cO_A^r,\C\cO_B^r\big)
&\cong&
\Hom_{\C U_J}\big(f([A])\C U_J,f([B])\C U_J\big)\\
&\cong&
\big(\C U_Jf([A])\big)\cap\big(f([B])\C U_J\big)\\
&\cong&
f(\C U_J[A])\cap f([B]\C U_J)\\
&\cong&
\C\cO_A^l\cap\C\cO_B^r\\
&\cong&
\C(\cO_A^l\cap\cO_B^r)
\qquad\text{by Remark~\ref{basicLA}}\\
&\cong&
\Hom_{\C U_J}\big(\C\cO_B^l,\C\cO_A^l\big),
\end{eqnarray*}
since $\C U_J$ is self-injective.

More precisely, if $[C]\in\cO_A^l\cap\cO_B^r$, then there exist $0\neq \lambda,\mu \in\C^*$ and  $x,y\in U_J$ such  that
\[
[C]=\lambda x[A]=\mu[B]y.
\]
The corresponding $\C U_J$-homomorphism
$
\C\cO_A^r\to\C\cO_B^r
$
is given by left multiplication by $\lambda x$. Since
$\cO_C^r=\cO_B^r$, this homomorphism is surjective, and it is injective because
$\lambda x$ is invertible. Similarly, right multiplication by $\mu y$
defines a $\C U_J$-isomorphism
$
\C\cO_B^l\to\C\cO_A^l,
$
which depends only on $[C]$. Moreover, $x$ and $y$ are unique up to multiplication by elements of the left stabilizer of $[A]$ and the right stabilizer of $[B]$, respectively.

This recovers, for the special case of pattern groups, the well-known fact (see, for example, [\cite{super}]) that the Andr\'e--Yan supercharacters are either equal or orthogonal. Finally, taking $A=B$, we obtain
\[
\End_{\C U_J}(\C\cO_A^r)
\cong
\End_{\C U_J}(\C\cO_A^l)
\]
as algebras with basis $\cO_A^l\cap\cO_A^r$.
\hfill$\square$
\end{Remark}

\begin{Cor}\label{iso orbits} Let $[A],[B]\in\V$. Then 
\[ 
\C\cOl_A\cong\C\cOl_B \text{ and } \C\cOr_A\cong\C\cOr_B\quad 
\text {if and only if } \cObi_A=\cObi_B.
\] 
Consequently, distinct biorbits of $\V$ afford orthogonal characters. 
\end{Cor} 
\begin{proof} Suppose that $\C\cOl_A\cong\C\cOl_B$. Then, by Remark~\ref{HomSpace}, there exists $u\in U$ such that $ u.[A]=[B]$, and hence $\cObi_A=\cObi_B$. Conversely, suppose that $\cObi_A=\cObi_B$. Then there exist $u,v\in U$ such that $ [A]=u.[B].v$, and therefore 
\[ 
\cOl_A = U.[A] = U.[B].v = \cOl_B.v. 
\] 
Hence right multiplication by $v$ induces an isomorphism from $\C\cOl_B$ onto $\C\cOl_A$. The corresponding statement for right orbits is proved analogously.  
\end{proof}

For later use we summarize some further facts about the orbits of the action of
$U_J$ on $\V_J$, which can be found in
[\cite{super}], [\cite{yan}], and [\cite{DG3}].

\begin{Facts}\label{gen orbit facts}
Let $[A],[B]\in\V_J$. Then the following hold.
\begin{enumerate}
\item[(1)]
Either the complex orbit modules $\C\cOr_A$ and $\C\cOr_B$ are isomorphic, or they have no irreducible constituent in common. The analogous statement holds for left orbit modules.

\item[(2)]
The one-sided orbits $\cOl_A$ and $\cOr_A$ have the same cardinality.

\item[(3)]
By a basic set-theoretic argument 
\[
|\cObi_A|
=
\frac{|\cOr_A||\cOl_A|}{|\cOr_A\cap\cOl_A|}
=
\frac{|\cOr_A|^2}{|\cOr_A\cap\cOl_A|}
=
\frac{|\cOl_A|^2}{|\cOr_A\cap\cOl_A|}.
\]
Therefore, using \ref{iso orbits}, we have 
\[
\C\cObi_A
\cong
\bigoplus_{\kappa\text{ copies}}\C\cOr_A
\cong
\bigoplus_{\kappa\text{ copies}}\C\cOl_A,
\]
where 
$\kappa=\frac{|\cOr_A|}{|\cOr_A\cap\cOl_A|} = \frac{|\cOl_A|}{|\cOr_A\cap\cOl_A|}$.
\hfill$\square$
\end{enumerate}
\end{Facts}

We now define the main objects of investigation in this paper.

\begin{Defn}
A subset $J$ of $\Phi^+$ is called \textbf{row closed} (respectively \textbf{column closed}) if
$(i,j)\in J$ implies $(i,l)\in J$ for all $i<l\le n$
(respectively $(k,j)\in J$ for all $1\le k<j$).
Equivalently, a row closed (respectively column closed) subset is obtained from $\Phi^+$ by removing all positions belonging to some rows (respectively some columns). \hfill $\square$
\end{Defn}

It is easy to see that every row closed or column closed subset of $\Phi^+$ is closed. The corresponding pattern subgroups $U_J$ are also called \textbf{row closed} or \textbf{column closed}, respectively.

\section{Switching Sides}

In this brief section we show that the problem of classifying the right orbit modules of row (respectively column) closed pattern subgroups is equivalent to the corresponding problem for the left orbit modules of column (respectively row) closed pattern subgroups. The key observation is that reflection at the anti-diagonal interchanges row closed and column closed patterns and reverses the order of matrix multiplication.

More precisely, define 
\[
i\longmapsto\bar i=n+1-i \text{ \,for\, } 1\le i\le n.
\]
This reflects the ordered set $\{1,\ldots,n\}$ in the midpoint $\frac{n+1}{2}$. Clearly,
\[
\bar{\bar i}=i
\qquad\text{and}\qquad
i\le j \Longleftrightarrow \bar j\le \bar i.
\]
Now define the mirror map
\begin{equation}\label{mirror}
\bar{\phantom{A}}:M_n(q)\longrightarrow M_n(q):
A\longmapsto \overline A=
\sum_{1\le i,j\le n}A_{ij}e_{\bar j\,\bar i}\quad \text{ for\, } A\in M_n(q).
\end{equation}
Let $J\subseteq\Phi^+$ be closed. Then
$
\bar J=\{(\bar j,\bar i)\mid (i,j)\in J\}\subseteq\Phi^+
$
is closed. Moreover, we have
\[
\overline{U_J}=U_{\bar J}, \quad\text{and}\quad
\overline{V_J}=V_{\bar J}.
\]

Let $A,B\in M_n(q)$ and $J\subseteq\Phi^+$.
Observe that
\[
\overline{AB}=\overline B\,\overline A,\quad
\overline{\pi_J(A)}=\pi_{\bar J}(\overline A)\quad\text{and}\quad
\overline{A^t}=\overline A^{\,t}.
\]
The mirror map extends naturally to $M_n(q)^*$ by setting
$
\overline{A^*}=\overline A^{\,*}.
$
Composing the mirror map with $\theta$ therefore yields a bijection
\[
\V_J\longrightarrow\V_{\bar J}:\overline{[A]}\longmapsto[\bar A]\in\V_{\bar J} \text{ for } [A]\in\V_J.
\]

Applying the mirror map to Lemma~\ref{dotstar}, we obtain, for
$u_1,u_2\in U_J$ and $A\in V_J$,
\begin{equation}
u_1.A^*.u_2
=
\overline{\overline{u_1.A^*.u_2}}
=
\overline{\bigl(\pi_{\bar J}(\bar u_2^{-t}\,\bar A\,\bar u_1^{-t})\bigr)^*}.
\end{equation}
Consequently,
\begin{equation}
u_1.[A].u_2
=
\overline{\bar u_2.[\bar A].\bar u_1}
=
\overline{[\bar u_2^{-t}\,\bar A\,\bar u_1^{-t}]}.
\end{equation}

We also observe that, by definition,
$
\overline{A}_{\bar j\bar i}=A_{ij}
$
for every $A\in V_J$. Let $(i,j)\in J$ and $\beta\in\F_q$. Applying the mirror map to Corollary~\ref{restricted action}, we obtain
\[
\overline{[A]x_{ij}(\beta)}
=
\theta(-\beta A_{ij})\,
\overline{[A].x_{ij}(\beta)}
=
\theta(-\beta\overline{A}_{\bar j\bar i})\,
x_{\bar j\bar i}(\beta).\overline{[A]}
=
x_{\bar j\bar i}(\beta)\overline{[A]}.
\]
Similarly,
\[
\overline{x_{ij}(\beta)[A]}
=
\overline{[A]}x_{\bar j\bar i}(\beta),
\]
for every $A\in V_J$.

Since every element of $U_J$ can be written as a product of root subgroup elements, we conclude:

\begin{Lemma}\label{mirrorlemma}
Let $J\subseteq\Phi^+$ be closed,  and let
$A\in V_J$. Then
$
\overline{u_1[A]u_2}
=
\bar u_2[\bar A]\bar u_1,
$ for all $u_1,u_2\in U_J$.
\hfill$\square$
\end{Lemma}

\begin{Prop}\label{switch}
Let $J\subseteq\Phi^+$ be closed. Then classifying the left (respectively right) $U_J$-orbits on $\V_J$ and the characters afforded by the corresponding orbit modules is equivalent to classifying the right (respectively left) $U_{\bar J}$-orbits on $\V_{\bar J}$ and the characters they afford.
\hfill$\square$
\end{Prop}

Let $J$ be row closed. By Proposition~\ref{switch},
the classification of the left $U_J$-orbits on $\V_J$ is equivalent to the
classification of the right $U_{\bar J}$-orbits on $\V_{\bar J}$, where
$\bar J$ is column closed. We may therefore apply the results of
[\cite{DG3}] and translate them, via the mirror map, into corresponding
results for the left orbits of $U_J$.

As will become apparent later, the classification of the left orbits of
$U_J$ on $\V_J$ is the key step in the classification of the right orbits. Furthermore, 
applying the mirror map once more, one obtains then the corresponding results for
the left orbits of column closed pattern subgroups of $U_n(q)$. 

Throughout
the remainder of the paper, we therefore assume that $J\subseteq\Phi^+$ is
row closed and study the action of $U_J$ on $\V_J$ from both sides.

\section{Left Orbits}

In [\cite{DG3}] we classified the right orbits of column closed pattern groups by means of so-called templates and determined which of the corresponding orbit modules are isomorphic and which afford orthogonal characters. In this section we transfer the basic framework of these results to row closed pattern groups using the mirror map of Section~\ref{mirror}. The reader may verify that the notions introduced here for left orbits are precisely the images, under the mirror map, of the corresponding constructions for right orbits of column closed pattern groups developed in Section~4 of [\cite{DG3}]. Since it will always be clear from the context whether we are considering the left action of a row closed pattern group or the right action of a column closed pattern group, we use the same notation in both settings.

\medskip

Henceforth we fix a row closed subset $J\subseteq\Phi^+$ and write $U=U_J$. To simplify notation, we omit the index ``$J$'' whenever no ambiguity can arise.

\begin{Defn} \label{normal} Let $1\leq j\leq n$. The column $j$, and all positions on it, are called $J$-\textbf{normal} or simply \textbf{normal}, if row $j$ is not contained in $J$. For $[A] \in\V$, the submatrix of $[A]$ consisting of all normal columns is called \textbf{normal kernel} of $[A]$ and is denoted by $\nker[A]$. 
\hfill$\square$
\end{Defn}

\begin{Remark}\label{normalfixedfromright} 
Let $(i,j)\in J$ be a normal position. Then Corollary~\ref{restricted action}
immediately implies that, for every $\alpha\in\F_q$, the root subgroup element
$x_{ij}(\alpha)$ acts from the left on each $[A]\in\V$ by the scalar
$\theta_{-\beta}(\alpha)$, where $\beta=A_{ij}$.

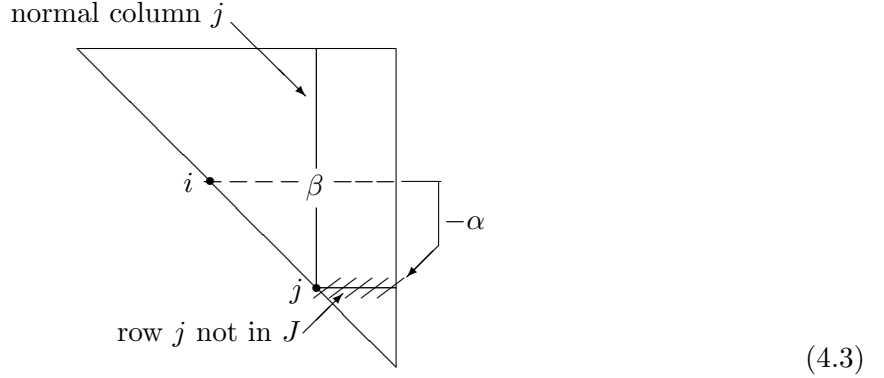
\begin{equation}\label{normalaction}
\begin{picture}(160,140)
\put(130,120){\line(0,-1){120}}
\put(130,120){\line(-1,0){120}}
\put(10,120){\line(1,-1){120}}

\put(50,66){\makebox{$i$}}
\put(60,70){\circle*{3}} 

\dottedline[$-$]{9}(110,70)(126,70)
\dottedline[$-$]{9}(61,70)(90,70)
\put(-15,130){\makebox{normal column $j$}}
\put(25,10){\makebox{row $j$ not in $J$}}
\put(95,13){\vector(1,1){15}}
\put(68,130){\vector(1,-1){28}}

\put(100,30){\circle*{3}} 
\put(90,26){\makebox{$j$}}

\dottedline{.2}(101,30)(130,30)
\multiput(99,26)(6,0){5}{\line(4,3){10}}

\dottedline{.2}(100,30)(100,60)
\dottedline{.2}(100,75)(100,120)

\put(96,65){\makebox{$\beta$}}

\put(147,70){\line(-1,0){15}}
\put(146,70){\line(0,-1){24}}
\put(146,46){\vector(-1,-1){12}}
\put(148,52){\makebox{$-\alpha$}}

\end{picture}
\end{equation}

Normality of columns has a different consequence for the right action of $U$.
Indeed, the entries
$
\{A_{sj}\mid 1\leq s<j\}
$
of a normal column $j$ of $[A]\in\V$ cannot be changed by right multiplication with any element of $U$. Consequently, the entries in all normal positions, and hence the entire normal kernel, are constant on the right orbit $\cO_A^r$.
To see this, observe that a root subgroup $X_{ab}$ acts nontrivially on column $j$ only if $a=j$. However, if row $j$ does not belong to $J$, then $X_{jb}\not\subseteq U_J$ by definition.
\hfill$\square$
\end{Remark}

\begin{Defn}\label{flip}
For $(i,j)\in J\subseteq\Phi^+$, the \textbf{hook} centered at $(i,j)$ is
\[
h_{ij}=h_{ij}^a\cup h_{ij}^l\cup\{(i,j)\},
\]
where
\[
h_{ij}^a=\{(i,k)\in J\mid i<k<j\}
\]
is the \textbf{hook arm},
\[
h_{ij}^l=\{(k,j)\in J\mid i<k<j\}
\]
is the \textbf{hook leg}, and $(i,j)$ is the \textbf{hook center}.

The \textbf{reduced hook} $\widehat h_{ij}$ is obtained from $h_{ij}$ by removing the hook center and all normal positions. Similarly, the \textbf{reduced hook arm} $\widehat h_{ij}^a$ is obtained from $h_{ij}^a$ by removing all normal positions.

There is an obvious involutory bijection
\[
\f=\f_{ij}:\widehat h_{ij}\longrightarrow\widehat h_{ij},
\]
which maps $(i,k)$ to $(k,j)$,
whenever $i<k<j$ and row $k$ is contained in $J$. Indeed, if row $k$ is not contained in $J$, then $(k,j)\notin J$, while $(i,k)$ is a normal position and hence does not belong to $\widehat h_{ij}^a$. Thus $\f^2=1$, and $\f$ induces a bijection between the hook leg $h_{ij}^l$ and the reduced hook arm $\widehat h_{ij}^a$. We call $\f$ the \textbf{flip map centered} at $(i,j)$.
\hfill$\square$
\end{Defn}

The following illustration shows the flip map. If row $k$ is contained in $J$,
then $(i,k)$ is mapped to $(k,j)$, and vice versa. If row $k$ is not contained
in $J$, then $(i,k)$ is a normal position and therefore does not belong to
$\widehat h_{ij}^a$:

\begin{equation}\label{picnormalbijection}
\begin{picture}(200,140)
\put(95,85){\line(0,-1){65}}
\put(95,85){\line(-1,0){65}}
\put(0,115){\line(1,-1){125}}

\put(95,85){\circle*{3}}
\put(30,85){\circle*{3}}
\put(95,20){\circle*{3}}
\put(95,40){\circle*{3}}
\put(75,85){\circle*{3}}
\put(75,40){\circle*{3}}
\put(97,44){\makebox{\small$(k,j)$}}
\put(90,90){\makebox{$(i,j)$}}
\put(52,74){\makebox{\small$(i,k)$}}
\put(22,80){\makebox{$i$}}
\put(86,15){\makebox{$j$}}
\put(66,35){\makebox{$k$}}

\put(75,40){\line(1,0){50}}

\put(75,40){\line(0,1){75}}

\put(125,115){\line(0,-1){125}}
\put(125,115){\line(-1,0){125}}

\put(150,115){\vector(-1,-1){52}}
\put(152.5,118){$h_{ij}^l$}

\put(86,133.5){\vector(-1,-1){47}}
\put(81,137){$h_{ij}^a$}

\end{picture}
\end{equation}
\smallskip

In [\cite{DG3}, 4.18] we classified the right orbits of column closed pattern subgroups of $U_n(q)$ by means of so-called templates, obtained from linear characters by removing as many nonvanishing entries as possible using restricted column operations. We now introduce the corresponding notion for the left action of row closed pattern subgroups, again called \textbf{templates}:

\begin{Defn}\label{template}
\begin{enumerate}
\item[i)]
A linear character $[A]\in\V$ is called a {\bf (left) template} if, for every nonzero row $\fr_i(A)$ of $A$, all entries
$A_{kj}$ with $i<k<j$ vanish, where $A_{ij}$ is the rightmost nonzero entry of row $i$. Equivalently, all entries on the hook leg $h_{ij}^l$ below the rightmost nonzero entry of each nonzero row vanish.

The position $(i,j)\in J$ is called the {\bf main condition} of row $i$. The set of all main conditions of the template $[A]$ is denoted by $\p[A]$. Since main conditions are rightmost nonzero entries of rows, they necessarily lie in distinct columns.

A subset $\p\subseteq J$ is called a {\bf set of main conditions} if $\p=\p[A]$ for some template $[A]\in\V$. The set
\[
\main[A]=\{(A_{ij},(i,j))\mid(i,j)\in\p[A]\}
\]
records the values of the main conditions together with their positions. The hooks centered at main conditions are called {\bf main hooks}, and intersections of main hooks are called {\bf main hook intersections}.

\item[ii)]
The positions of a template $[A]$ lying in the same row as a main condition, strictly to its left and not below another main condition, are called {\bf supplementary conditions}. Their set is denoted by $\suppl(\p)$, where $\p=\p[A]$. Equivalently, $\suppl(\p)$ consists precisely of the positions on main hook arms that are not main hook intersections. By abuse of terminology, the entries of $[A]$ at these positions are also called {\bf supplementary conditions}. Apart from the main entries, these are the only entries of $A$ that may be nonzero.

\item[iii)]
Let $[A]$ be a template and $\p=\p[A]$. The positions in $\suppl(\p)$ belonging to normal columns are called {\bf normal supplementary conditions}; their set is denoted by
$
\bfn=\bfn(\p).
$
The complementary set
$
\bfy=\bfy(\p):=\suppl(\p)\setminus\bfn(\p)
$
consists of the {\bf nonnormal supplementary conditions}, or  {\bf $y$-conditions} for short. Finally, $[A]$ is called a {\bf normal template} if all of its nonzero supplementary conditions are normal.
\hfill$\square$
\end{enumerate}
\end{Defn}

\begin{Example}\label{hookintersection}
Let  $J=\Phi^+\setminus \{\text{row $k$}\}$, let $[A]$ be a template in $\V$ with main condition set $\p=\{(i,l), (j,m)\}$ and let $A_{il} = z_1$ and $A_{jm}=z_2$.
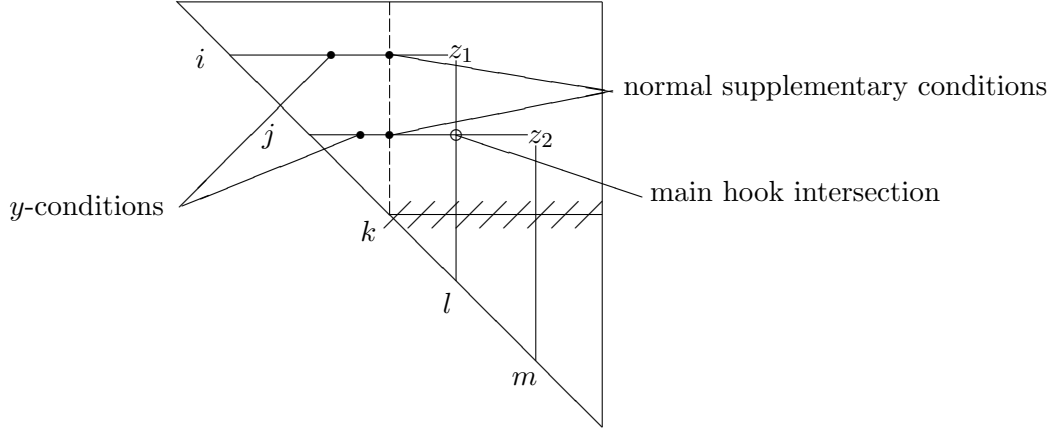
\begin{center}
\begin{picture}(250,130)
\put(-63,80){$y$-conditions}
\put(58,140){\circle*{3}}
\put(58,140){\line(-1,-1){57}}
\put(69,110){\circle*{3}}
\put(69,110){\line(-5,-2){67}}

\put(0,160){\line(1,0){160}}
\put(0,160){\line(1,-1){160}}
\put(160,160){\line(0,-1){160}}

\put(20,140){\line(1,0){82}}
\put(102,139){$z_1$}
\put(105,138){\line(0,-1){83}}

\put(7,135){$i$}
\put(100,43){$l$}

\put(69,70){$k$}

\put(50,110){\line(1,0){82}}
\put(132,107){$z_2$}
\put(135,106){\line(0,-1){81}}

\put(32,107){$j$}
\put(126,16){$m$}
\put(105,110){\circle{4}}
\put(105,110){\line(3,-1){70}}
\put(178,85){main hook intersection}

\put(80,80){\line(1,0){80}}
\multiput(78,75)(9,0){9}{\line(1,1){10}}
\multiput(80,80)(0,6.25){13}{\line(0,1){5}}

\put(80,110){\circle*{3}}
\put(80,140){\circle*{3}}
\put(81,110){\line(5,1){83}}
\put(81,140){\line(6,-1){83}}
\put(168,125){normal supplementary conditions}


\end{picture}
\end{center}
Column $k$ is normal, and $(j,l)$ is the main hook intersection.
Furthermore, all positions in row $j$ strictly to the left of $z_2$, except
for $(j,l)$, together with all positions in row $i$ strictly to the left of
$z_1$, are the supplementary conditions. Among these, $(i,k)$ and $(j,k)$ are
normal supplementary conditions, whereas all remaining supplementary
conditions are $y$-conditions.\hfill$\square$
\end{Example}

\begin{Procedure}\label{maketemplate}
We now describe, with the aid of Illustration~\ref{pic4-3} below, how
restricted row operations can be used to eliminate entries from
$[A]\in\V$ and thereby produce a template in the left orbit
$\cO_A^l$.

\begin{center}
\begin{equation}\label{pic4-3}
\begin{picture}(220, 110)


\put(50,0){\line(0,1){130}}
\put(50,0){\line(-1,1){130}}
\put(50,130){\line(-1,0){130}}


\put(-56,84){$i$}
\put(-45,95){\circle*{4}}
\dottedline[\tiny x]{5}(-40,95)(-22,95)
\put(-21,93){$\scriptstyle A_{ik}$}
\dottedline[\tiny x]{5}(-5,95)(5,95)
\put(7,93){$\scriptstyle A_{ij}$}
\put(21,93){$\scriptstyle 0 \ldots\ldots  0$}


\put(-21,49){$k$}
\put(-10,60){\circle*{4}}
\dottedline[\tiny x]{5}(-5,60)(5,60)
\put(7,58){$\scriptstyle A_{kj}$}
\dottedline[\tiny x]{5}(25,60)(44,60)


\put(04,24){$j$}
\put(15,35){\circle*{4}}


\multiput(-10,65)(0,6){4}{\line(0,1){5}}
\multiput(15,65)(0,6){4}{\line(0,1){5}}
\multiput(15,38)(0,6){3}{\line(0,1){5}}


\put(50,96){\line(1,0){10}}
\put(59,96){\line(0,-1){26}}
\put(59,70){\vector(-1,-1){10}}
\put(64,80){$-A_{ij}^{-1}A_{kj}$}


\put(135,56){$ = $}
\put(-115,56){$x_{ik}(A_{ij}^{-1}A_{kj}) . $}



\put(270,0){\line(0,1){130}}
\put(270,0){\line(-1,1){130}}
\put(270,130){\line(-1,0){130}}


\put(164,84){$i$}
\put(175,95){\circle*{4}}
\dottedline[\tiny x]{5}(180,95)(198,95)
\put(199,93){$\scriptstyle A_{ik}$}
\dottedline[\tiny x]{5}(215,95)(225,95)
\put(227,93){$\scriptstyle A_{ij}$}
\put(241,93){$\scriptstyle 0 \ldots\ldots  0$}


\put(199,49){$k$}
\put(210,60){\circle*{4}}
\dottedline[\tiny x]{5}(215,60)(225,60)
\put(233,58){$\scriptstyle 0$}
\dottedline[\tiny x]{5}(245,60)(264,60)


\put(224,24){$j$}
\put(235,35){\circle*{4}}


\multiput(210,65)(0,6){4}{\line(0,1){5}}
\multiput(235,65)(0,6){4}{\line(0,1){5}}
\multiput(235,38)(0,6){3}{\line(0,1){5}}

\end{picture}
\end{equation}
\end{center}

 If $A$ is the zero matrix, then $[A]$ is the unique template in
$\cO_A^l=\{[A]\}$. Hence assume that $A\neq0$. Let $i$ be the index of a
nonzero row of $A$, and let $j$ be maximal with $i<j\le n$ and
$A_{ij}\neq0$, so that $A_{ij}$ is the rightmost nonzero entry of row
$i$. Let $i<k<j$ and define
\[
[B]=x_{ik}(A_{ij}^{-1}A_{kj}).[A]\in\cO_A^l.
\]
Then $B$ agrees with $A$ except possibly in row $k$, at positions on or
to the left of $(k,j)$. Moreover, $B_{kj}=0$, as can be seen from
Illustration~\ref{pic4-3}. Repeating this procedure for every entry
$A_{kj}$ in column $j$ with $i<k<j$, we obtain a linear character
$[C]\in\cO_A^l$ such that $C_{ij}=A_{ij}$, every entry of column $j$
below $(i,j)$ vanishes, and $C$ differs from $A$ only at positions
strictly southwest of $(i,j)$.

Proceeding in this way through all nonzero rows from top to bottom,
starting with the highest nonzero row, produces the desired template in
$\cO_A^l$ (cf.~[\cite{DG3}, 4.9]).
\hfill$\square$
 \end{Procedure}

In [\cite{DG3}, 4.13], for column closed pattern groups $U_J$, we defined the
\textbf{places} $\Pl(\p)=\Pl(\p,J)$ associated with a set of main conditions
$\p$ to be the positions on the main hook arms. The corresponding notion for
row closed pattern groups is obtained by replacing hook arms with hook legs.

\begin{Defn}\label{lplaces}
Let $[A]\in\V$ be a template, and let $\p=\p[A]$. The positions lying on the hook legs
$h_{ij}^l$, where $(i,j)\in\p$, are called
$\pmb{\ell}$-\textbf{places} of $\p$. Their set is denoted by
$
\lPl(\p) = \lPl(\p,J),
$
or simply by $\lPl$ when $\p$ is fixed.
\hfill$\square$
\end{Defn}

\begin{Remark}\label{mirror rl}
Recall that the mirror map \ref{mirror} reflects matrices across the antidiagonal, thereby interchanging rows and columns. It therefore induces natural bijections between the various objects associated with the row closed pattern subgroup $U_J$ of $U_n(q)$ and those associated with its mirror group $U_{\bar J}$.

In particular, it yields a bijection from $\V_J$ onto $\V_{\bar J}$, taking sets of main conditions $\p\subseteq J$ to the corresponding sets of main conditions $\overline\p\subseteq\bar J$ (see [\cite{DG3}, 4.8]) and main hooks for $\p$ to the corresponding main hooks for 
$\overline\p$, while interchanging hook legs and hook arms. Moreover, the mirror image of a template $[A]\in\V_J$ is precisely the template 
$[\bar A]\in\V_{\bar J}$ defined in [\cite{DG3}, 4.8]. Furthermore, the mirror map sends the left orbit
$\cOl_A\subseteq\V_J$ bijectively onto the right orbit
$\cOr_{\bar A}\subseteq\V_{\bar J}$, and maps the set of $\ell$-places
$\lPl(\p,J)$ bijectively onto $\Pl(\overline\p,\bar J)$.
\hfill$\square$
\end{Remark}

Theorem [\cite{DG3}, 4.18] asserts that every right orbit in $\V_{\bar J}$
contains a unique (right) template. Together with Remark~\ref{mirror rl},
this immediately yields:

\begin{Theorem}\label{unique one}
Each left $U_J$-orbit of $\V_J$ contains a unique template. Consequently,
the templates classify the left $U_J$-orbits on $\V_J$.
\hfill$\square$
\end{Theorem}

We again suppress the index $J$ until further notice.

Given a template $[A]\in\V$, Remark~\ref{mirror rl} and [\cite{DG3}, 4.19] immediately yield the following description of the elements in its left orbit:

\begin{Prop}\label{placebij}
Let $[A]\in\V$ be a template with $\p[A]=\p$. Then, for every choice of
scalars $\beta_{ij}\in\F_q$, $(i,j)\in\lPl(\p)$, there exists a unique
$[B]\in\cOl_A$ satisfying
\[
B_{ij}=\beta_{ij}\qquad\text{for all }(i,j)\in\lPl(\p).
\]
Consequently,
$
|\cOl_A|=q^{|\lPl(\p)|},
$
and, for every $[B]\in\cOl_A$, the entries $B_{kl}$ with
$(k,l)\in J\setminus\lPl(\p)$ are linear functions of the entries
$B_{ij}$, $(i,j)\in\lPl(\p)$.

If, in addition, $[A]$ is a normal template, then every nonzero nonnormal
entry of $[B]$ lies on a main hook leg. In particular, all entries of
$[B]$ lying above a nonnormal main condition in its column vanish.
\hfill$\square$
\end{Prop}

\section{Special Right Orbits}

In the previous section, we described the left $U_J$-orbits on $\V_J$
(and, equivalently, the right $U_{\bar J}$-orbits on $\V_{\bar J}$) in
terms of templates. There does not appear to be a corresponding notion of
right templates that parametrizes the right orbits. Nevertheless, as we
shall see, the classification of the left orbits, together with a counting
argument, provides a route to the classification of the right orbits.

Corollary~\ref{iso orbits} reduces the classification of the one-sided
orbits to two problems. First, the decomposition of $\V_J$ into biorbits
provides, by choosing one linear character $[A]\in\cObi$ from each
biorbit, a complete set of representatives of the pairwise nonisomorphic
orbit modules $\C\cOl_A$ (equivalently $\C\cOr_A$). It therefore remains
to decompose each biorbit into its one-sided orbits.

We begin by studying the biorbits of $U_J$ acting on $\V_J$.
We again suppress the index $J$ until further notice.

Let $[A]\in\V$. By Theorem~\ref{unique one}, the distinct left orbits
contained in the biorbit $\cObi_A$ are in one-to-one correspondence with
the templates contained in $\cObi_A$. Thus, to determine all left orbits
contained in $\cObi_A$, it suffices to determine these templates.
Clearly,
\begin{equation}\label{biorbitdec}
\cObi_A = U.[A].U =
\begin{cases}
\cOl_A.U = \displaystyle\bigcup_{[B]\in\cOl_A}\cOr_B,\\[2mm]
U.\cOr_A = \displaystyle\bigcup_{[C]\in\cOr_A}\cOl_C.
\end{cases}
\end{equation}

Hence, to classify the left orbits contained in $\cObi_A$, it suffices to
determine, for each $[B]\in\cOr_A$, the unique template contained in the
left orbit $\cOl_B$, and then identify which of these templates coincide.
Similarly, to classify the right orbits contained in $\cObi_A$, it suffices
to investigate the right orbits $\cOr_B$ with $[B]\in\cOl_A$.

Observe that, when studying $\cObi_A$, we may always assume that $[A]$
itself is a template, since every biorbit contains a left orbit and,
by Theorem~\ref{unique one}, each left orbit contains a unique template.
In Proposition~\ref{normaltempl} below, we shall show that we may, in fact,
assume that $[A]$ is a normal template.

\begin{Defn}\label{quasitemplate}
A linear character $[B]\in\V$ is called {\bf special} if the rightmost
nonzero entries of the nonzero rows of $B$ lie in pairwise distinct
columns. Orbits consisting of special linear characters are called
{\bf special} as well.

If $[B]\in\V$ is special, then the positions of the rightmost nonzero
entries of the nonzero rows of $B$ form the set
$\p[B]=\p$ of {\bf main conditions}, thereby extending the definition of
main conditions for templates given in Definition~\ref{template}(i). We
set
\[
\main[B]=\{(B_{ij},(i,j))\mid(i,j)\in\p[B]\}.
\]

Thus a special linear character differs from a template only in that
entries at main hook intersections are allowed to be nonzero.

Finally, if $[B]\in\V$ is special and all entries at the
$y$-conditions of $\p[B]$ vanish, then $[B]$ is called
{\bf extraspecial}.
\hfill$\square$
\end{Defn}

Proposition~\ref{placebij} provides a precise description of the left
orbit determined by a template. As noted at the beginning of this
section, there does not appear to be a corresponding notion of templates
for right orbits. Nevertheless, the right orbits generated by templates,
or more generally by special linear characters, will play a crucial role
in the classification of the right orbit $\C U_J$-modules. We now
describe the elements of these right orbits:

\begin{Defn}\label{rplaces}
Let $[A]\in\V$ be special, and let $\p=\p[A]$.
The positions $(r,s)\in\widehat h_{ij}^a$, where $(i,j)\in\p$, are called
\textbf{$r$-places} of $\p$. Their set is denoted by $\rPl(\p)$.
Equivalently, the $r$-places are precisely the nonnormal positions on
nonzero rows lying strictly to the left of the main conditions.
For $(i,j)\in\p$, the set of $r$-places contained in the reduced hook arm
$\widehat h_{ij}^a$ is denoted by $\rPl(i,j)$.
\hfill$\square$
\end{Defn}

\begin{Prop}\label{rplacebij}
Let $[A]\in\V$ be special with $\p[A]=\p$. Then, for any choice of
elements $\beta_{rs}\in\F_q$, $(r,s)\in\rPl(\p)$, there exists a unique
$[C]\in\cOr_A$ such that
\begin{enumerate}
\item[(i)] $C_{rs}=\beta_{rs}$ for all $(r,s)\in\rPl(\p)$;
\item[(ii)] $C_{ij}=A_{ij}$ for all main conditions $(i,j)\in\p$ and all
normal positions $(i,j)\in J$;
\item[(iii)] $C_{rs}=0$ for all remaining positions $(r,s)\in J$.
\end{enumerate}
Conversely, every element $[C]\in\cOr_A$ arises in this way.
Moreover, every element of $\cOr_A$ is special and satisfies
$\main[C]=\main[A]$.
In particular, the right orbit $\cOr_A$ is special, and the main
conditions together with their entries are constant on $\cOr_A$.
\end{Prop}

\begin{proof} 
The basic step consists of using the restricted column operations from
Corollary~\ref{restricted action} to prescribe arbitrary entries
$\beta\in\F_q$ at nonnormal positions $(r,s)\in J$ lying in a row
containing a main condition $(r,t)\in\p$, where $r<s<t$.
Observe that restricted column operations preserve the zero rows and the
rightmost nonzero entry of every nonzero row.

Set $0\neq A_{rt}=\alpha$ and suppose that $A_{rs}=\eta\in\F_q$.
Since $(r,s)$ is nonnormal, row $s$ is contained in $J$, and therefore
$X_{st}\le U$. Acting on $[A]$ from the right by
$x_{st}(\gamma)$, where
$\gamma=\alpha^{-1}(\eta-\beta)\in\F_q$, yields
$
[C]=[A].x_{st}(\gamma),
$
and hence $C_{rs}=\beta$, as illustrated below:
\bigskip

\begin{equation}\label{ycondpic}
\begin{picture}(150,110)
\put(130,100){\line(0,-1){130}}
\put(130,100){\line(-1,0){130}}
\put(0,100){\line(1,-1){130}}


\put(90,10){\circle*{3}} 
\put(80,5){\makebox{$c$}}

\put(10,90){\circle*{3}} 
\put(2,86){\makebox{$r$}}
\put(10,90){\line(1,0){27}}
\put(42,90){\line(1,0){27}}
\put(38,88){\makebox{$\eta$}}
\put(67,88){\makebox{$\alpha$}}
\put(70,30){\circle*{3}} 
\put(58,26){\makebox{$t$}}
\put(70,30){\line(0,1){40}}
\put(69,78){\line(0,1){10}}

\put(40,60){\circle*{3}} 
\put(31,55){\makebox{$s$}}
\put(40,60){\line(1,0){91}}
\put(40,60){\line(0,1){27}}


\put(25,75){\circle*{3}} 
\put(17,71){\makebox{$b$}}
\put(25,75){\line(1,0){42}}
\put(72,75){\line(1,0){15}}
\put(67,72){\makebox{$\zeta$}}
\put(87,72){\makebox{$\rho$}}
\put(90,10){\line(0,1){60}}

\put(148,88){\makebox{main conditions}}
\put(145,90){\vector(-4,-1){45}}
\put(145,90){\vector(-4,0){63}}

\put(145,73){\vector(-1,-1){13}}
\put(148,70){\makebox{row $s\subseteq J$}}


\put(69,103){\line(0,1){10}}
\put(70,113){\line(-1,0){19}}
\put(51,113){\vector(-1,-1){10}}
\put(50,116){\makebox{$-\gamma$}}

\put(-75,45){\makebox{$[A].x_{st}(\gamma):$}}

\put(148,30){\makebox{$(r,s)=\f(s,t)$ wrt hook $h_{rt}$}}

\end{picture}
\end{equation}
\bigskip
\bigskip

 We remark that this operation changes the entry $A_{bs}$ to
$
C_{bs}=A_{bs}-\gamma\zeta.
$

Observe that, when prescribing the entries on a fixed reduced hook arm
$\widehat h_{rt}^a$, the order in which the restricted column operations
are performed is irrelevant, since they amount to right multiplication by
an element of the abelian subgroup
$
\prod_{r<s<t}X_{st}.
$

Furthermore, apart from the prescribed entry at $(r,s)$, the linear
characters $[A]$ and $[A].x_{st}(\gamma)$ can differ only at positions
$(b,s)$ with $r<b<s$ for which the entry
$\zeta=A_{bt}$ to the right of $(b,s)$ is nonzero. Such a position
$(b,t)$ lies on the reduced hook arm $h_{bc}^a$ of the main condition
$(b,c)\in\p$ for some $r< b<t$.

Order the main conditions as
\begin{equation}\label{p order 1}
\p=\{(r_1,t_1),\ldots,(r_k,t_k)\},
\qquad
t_1<t_2<\cdots<t_k.
\end{equation}
We process the corresponding reduced hook arms in this order going from left to right. First we
prescribe the entries at the nonnormal positions in row $r_1$ to the left
of the main condition $(r_1,t_1)$. This may change entries at nonnormal
positions $(b,c)\in J$ with $b<t_1$. However, all such positions lie in
rows $b\in\{r_2,r_3,\ldots,r_k\}$, which have not yet been processed and
will be treated by later applications of the move illustrated in
\ref{ycondpic}.

Now assume inductively that the prescribed entries have already been
assigned on the reduced hook arms
$\widehat h_{r_\nu t_\nu}^a$ for
$1\le\nu<\mu\le k$.
Since $t_\nu<t_\mu$, the entry in position
$(r_\nu,t_\mu)$ vanishes, because $(r_\nu,t_\nu)$ is the rightmost
nonzero entry of row $r_\nu$. Consequently, the restricted column
operations used to prescribe the entries on
$\widehat h_{r_\mu t_\mu}^a$ do not affect row $r_\nu$.
Thus, once a reduced hook arm has been processed, it remains unchanged.

By induction, every entry outside the main conditions and the normal
supplementary conditions can therefore be prescribed arbitrarily by right
multiplication with a suitable element of $U$.

Finally, restricted column operations preserve the rightmost nonzero entry
of every nonzero row and do not alter entries outside the reduced hook
arms. Hence every element of $\cOr_A$ arises in the manner described
above, and the remaining assertions follow immediately.
\end{proof}

\begin{Remark}\label{OrOl bij}
Let again $[A]\in\V$ be special. In the setting of
Proposition~\ref{rplacebij}, the flip maps
$\f_{ij}:h_{ij}^l\rightarrow\widehat h_{ij}^a$, where $(i,j)$ runs through
$\p=\p[A]$, combine to a bijection
$
\f:\lPl(\p)\longrightarrow\rPl(\p).
$
Hence
$
|\cOl_A|
=
q^{|\lPl(\p)|}
=
q^{|\rPl(\p)|}
=
|\cOr_A|.
$
Furthermore, choosing $\beta_{rs}=0$ for all
$(r,s)\in\rPl(\p)$ yields an extraspecial linear character
$[E^A]\in\cOr_A$, which in addition vanishes at all nonnormal main hook
intersections. Consequently, $[E^A]$ is a normal template if and only if
$A_{rs}=0$ for every normal main hook intersection $(r,s)\in J$.
\hfill$\square$
\end{Remark}

Taking Remark~\ref{normalfixedfromright} into account, we obtain:

\begin{Cor}\label{extraspecial exists}
Every special right orbit $\cO\subseteq\V$ contains a unique
extraspecial linear character $[E]=[E^A]$ that vanishes at all nonnormal
main hook intersections. In particular,
$
\cO=\cOr_E.
$
Moreover, $[E]$ is a normal template if and only if every linear character
in $\cO$ vanishes at the normal main hook intersections.
\hfill$\square$
\end{Cor}

Proposition~\ref{rplacebij} shows that the main conditions together with their values are constant on special right orbits. In fact, a much stronger statement holds: they are constant on the union of all special right orbits contained in a given biorbit of~$\V$. The key observation is the following immediate consequence of Procedure~\ref{maketemplate}. 

\begin{Lemma}\label{mainpreserved 1} Let $[B]\in\V$ be special, and let $[T]$ be the unique template in $\cOl_B$. Then $[T]$ can be obtained from $[B]$ by a sequence of restricted row operations such that every intermediate linear character $[S]$ is special and satisfies 
$\main[B]=\main[S]$. Consequently $\main[\,\,\,]$ is constant on the special linear characters contained in the left orbit $\cOl_B = \cOl_T$.
\hfill$\square$ 
\end{Lemma}

\begin{Example}\label{make template} We illustrate this with an example:
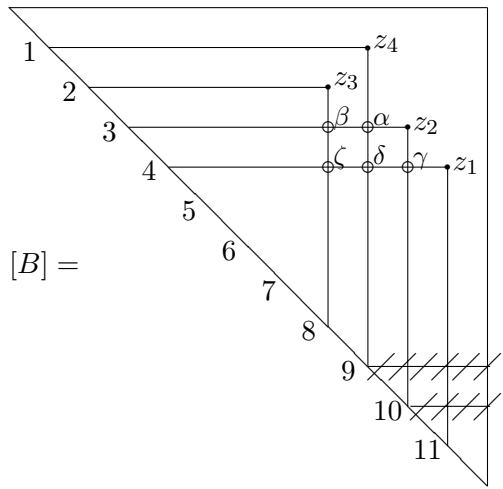
\quad
\begin{center}
\begin{picture}(350,210)

\put(0,100){$[B]=$}
\put(0,200){\line(1,0){180}}
\put(0,200){\line(1,-1){180}}
\put(180,200){\line(0,-1){180}}
\put(5,180){1}
\put(20,165){2}
\put(35,150){3}
\put(50,135){4}
\put(65,120){5}
\put(80,105){6}
\put(95,90){7}
\put(110,75){8}
\put(125,60){9}
\put(137,45){10}
\put(152,30){11}
\put(135,65){\line(1,0){46}}
\multiput(135,60)(8,0){6}{\line(1,1){10}}
\put(151,50){\line(1,0){30}}
\multiput(151,45)(8,0){4}{\line(1,1){10}}

\put(60,140){\line(1,0){105}}
\put(165,140){\line(0,-1){105}}
\put(165,140){\circle*{2}}
\put(167,139){$z_1$}
\put(152,141){\small $\gamma$}
\put(150,140){\circle{4}}
\put(137,141){\small $\delta$}
\put(135,140){\circle{4}}
\put(122,141){\small $\zeta$}
\put(120,140){\circle{4}}
\put(45,155){\line(1,0){105}}
\put(150,155){\line(0,-1){105}}
\put(150,155){\circle*{2}}
\put(152,154){$z_2$}
\put(135,155){\circle{4}}
\put(137,156){\small $\alpha$}
\put(120,155){\circle{4}}
\put(122,156){\small $\beta$}

\put(30,170){\line(1,0){90}}
\put(120,170){\line(0,-1){90}}
\put(120,170){\circle*{2}}
\put(122,171){$z_3$}

\put(15,185){\line(1,0){120}}
\put(135,185){\line(0,-1){120}}
\put(135,185){\circle*{2}}
\put(137,185){$z_4$}

\put(230,150){$\p=\{(4,11), (3,10), (2,8), (1,9)\}$}

\put(230,130){columns 9 and 10 are normal}

\put(230,100){main hook intersections:}

\put(230,80){$\{(3,9),(3,8),(4,10),(4,9),(4,8) \}$}
\end{picture}
\end{center}
Since $[B]$ is assumed to be special the only nonzero entries of $B$ are in rows $1,2,3$ and $4$ to the left and on the main conditions $z_1,z_2,z_3$,
 and $z_4$ respectively. There are nonzero entries $\alpha,\gamma,\delta$  at normal and $\beta,\zeta$ at nonnormal main hook intersections.
 All further nonzero values of $B$  are supplementary conditions not indicated in the illustration. To obtain zero values at the main hook intersections we act by suitable elements of the  root subgroups $X_{1,3},X_{2,3},X_{3,4},X_{1,4},X_{2,4}$ in that or any order compatible with the order of the hook arms top down.
 \hfill$\square$
\end{Example}

\begin{Prop}\label{mainpreserved} Let $\cO\subseteq\V$ be a biorbit, and let $[S],[T]\in\cO$ be special. Then $\main[S]=\main[T]$, and hence $\main[\,\,\,]$ is constant on the set of special linear characters contained in $\cO$. Moreover, if $[T]\in\cO$ is a template, then every special linear character in the left orbit $\cOl_T\subseteq\cO$ can be obtained from $[T]$ by a sequence of restricted row operations such that every intermediate linear character is again special. \end{Prop}

\begin{proof}
In view of Lemma~\ref{mainpreserved 1}, we may assume that $[T]$ is a
template. Since $[S]\in\cObi_T=\cO$, we find $x,y\in U$ such that
$[S]=x.[T].y$, and hence
\[
[\tilde S]=[S].y^{-1}=x.[T]\in\cOl_T
\]
is special with $\main[\tilde S]=\main[S]$ by
Proposition~\ref{rplacebij}. Moreover,
$\main[T]=\main[\tilde S]$ by Lemma~\ref{mainpreserved 1}, and therefore
$\main[S]=\main[T]$, as desired. Finally, reversing the sequence of
restricted row operations in Lemma~\ref{mainpreserved 1} proves the last
assertion.
\end{proof}

As an immediate consequence we obtain result [\cite{DG3}, 5.1]. Since every
biorbit $\cObi\subseteq\V$ contains a left orbit, it contains a template
$[A]$, and hence $\cObi=\cObi_A$. By Proposition~\ref{rplacebij}, the
right orbit $\cOr_A$ consists of special linear characters $[C]$
satisfying $C_{rs}=0$ for all normal main hook intersections. Moreover,
 $\cOr_A$ contains a unique
extraspecial linear character
$[E]\in\cOr_A\subseteq\cObi_A$ by Corollary~\ref{extraspecial exists}, which is a normal template. Thus we have
shown:

\begin{Cor}\label{normaltempl}
Every biorbit $\cObi$ in $\V$ contains normal templates.
\hfill$\square$
\end{Cor}

\begin{Notation} \label{Q sets} From now on we use the following notations:  
\begin{enumerate}
\item[(1)] $[A]\in\V$ is a normal template in $\V$.
\item[(2)] For $W\subseteq\V$  we denote the set of templates contained in $W$ by 
$\fT(W)$, of normal templates by $\fTn(W)$,
of special linear characters by $\fS(W)$ and of extraspecial linear characters by $\fE(W)$.  
For $[B]\in\V$ we replace the argument $\cObi_B$ by the index $B$ and write $\fT_B = \fT(\cObi_B)$
and similarly $\fTn_B,\fS_B$ and $\fE_B$. 
\item[(3)] For $[B]\in\V$ we abbreviate
\[
\cOrl_B=\cOr_B\cap\cOl_B
\]
and recall from Remark~\ref{HomSpace} that
$\cOrl_B$ determines a basis of
$\End_{\C U}(\C\cOr_B)$.
\item[(4)] For $[C]\in\fS_B$ let $[C^\fe]\in\fE_B$ denote the unique extraspecial linear character satisfying 
\[
C^\fe_{ij}=
\begin{cases}
C_{ij}, & \text{if }(i,j)\in\p[C]\cup \nker[C],\\
0, & \text{otherwise} 
\end{cases}
\]
If in addition $C_{rs} = 0$ for all normal main hook intersections, then $[C^\fe]$ is a normal template in $\fTn_B$.  
\item[(5)] $[T_B]\in\cOl_B$ denotes the unique template contained in the left orbit $\cOl_B$. 
\hfill$\square$
\end{enumerate}
\end{Notation}

By Theorem~\ref{unique one}, it suffices to determine all templates in
$\cObi_A$ in order to classify the left orbits contained therein. We now
show that it is in fact enough to construct the normal templates. Indeed,
all templates in $\cObi_A$ are then obtained simply by assigning arbitrary
values from $\F_q$ to the $y$-conditions of the normal templates.

\begin{Prop}\label{normalenough}
Let $[T]\in\cObi_A$ be a normal template. Then, for any choice of
$\beta_{ij}\in\F_q$ for $(i,j)\in\bfy=\bfy(\p)$, the linear character
$[C]\in\V$ defined by $C_{ij}=\beta_{ij}$ for all $(i,j)\in\bfy$ and
$C_{ij}=T_{ij}$ otherwise, is a template in $\fT(\cOr_T)$. As the
$\beta_{ij}$ vary over $\F_q$, one obtains all templates in
$\fT(\cOr_T)$. Thus, in particular,
$|\fT(\cOr_T)|=q^\upsilon$, where $\upsilon=|\bfy|$.
Moreover, the union
\begin{equation}\label{templates fT}
\bigcup_{[T]\in\fTn_A}\fT(\cOr_T)
\end{equation}
is disjoint and equals the set $\fT_A$ of all templates in $\cObi_A$.
\end{Prop}

\begin{proof} By construction, $[C]$ coincides with $[T]$ at the main conditions and on the normal columns, and has nonzero entries elsewhere only at positions on the main hook arms of $\p[T]=\p[A]$. Moreover, it vanishes at all main hook intersections. Therefore $[C]$ is a template. Now \ref{rplacebij} yields that 
$[C]\in\cOr_T$ and hence $[C]\in\fT(\cOr_T)$. 

Conversely, let $[D]\in\cOr_T$ be a template. Then
$\nker[D]=\nker[T]$, and hence $[D^\fe]=[T]$. In particular,
$[D]$ and $[T]$ coincide except possibly at positions on reduced main
hook arms that are not main hook intersections. Thus they differ only at
positions $(i,j)\in\bfy$, and hence $[D]$ is one of the linear
characters $[C]$ constructed above.

If $[C]\in\cObi_A$ is a template, then the right orbit
$\cOr_C$ contains a normal template
$[T]\in\fTn_A$ by Corollary~\ref{extraspecial exists}. Hence
$[C]\in\fT(\cOr_T)$, and therefore $[C]$ belongs to the union
\ref{templates fT}. Obviously, the subsets
$\fT(\cOr_T)$  have pairwise distinct normal
kernels. Hence the union is disjoint. 
\end{proof}

Recall that $[A]\in\V$ is a normal template by the general assumption i) in Notation \ref{Q sets}.

\begin{Lemma}\label{normtempl yconds}
Linear characters $[B]$ in $\cOl_A$ vanish at the $y$-conditions
$(i,j)\in\bfy(\p)$. Consequently, all special linear characters in
$\cOl_A$ are extraspecial.
\end{Lemma}

\begin{proof}
Let $(i,j)\in\bfy(\p)$. Since $(i,j)$ is a $y$-condition, it is
nonnormal by definition, and there is a main condition in row $i$ to the
right of $(i,j)$. Moreover, $(i,j)$ is not a main hook intersection.
Hence column $j$ is not normal and contains no main condition
$(a,j)$ with $a<i$. Since $[A]$ is a normal template, this implies
$A_{aj}=0$ for all $1\le a\le i$, and hence
$B_{aj}=0=B_{ij}$. The lemma follows.
\end{proof}

\begin{Theorem}\label{nker suff}
Let $[C]\in\cOl_A$ be special and suppose $[C]$ vanishes at normal main hook intersections. Then $[C]$ is extraspecial, and there exists 
$y\in U$ such that $[T]=[C].y\in\cObi_A$ is a normal template. Moreover, the correspondence $\cT_A:[C]\mapsto[T]$ defines a surjective map from the set of extraspecial  linear characters in $\cOl_A$ that vanish at all normal main hook intersections onto the set of normal templates in the biorbit $\cObi_A$.  
\end{Theorem}
\begin{proof}
There exists $z\in U$ such that $[C]=z.[A]$. Since $[C]$ is special
by assumption, it is extraspecial by Lemma~\ref{normtempl yconds}.
Hence, by Corollary~\ref{extraspecial exists}, there exists $y\in U$
such that
\[
[T]=[C].y=z.[A].y\in\cObi_A
\]
is the unique normal template in $\cOr_C$. Therefore, the
correspondence $\cT_A$ is a well-defined map.

Now let $[T]\in\cObi_A$ be a normal template. Then $[T]$ is
extraspecial and, as a template, vanishes at all main hook
intersections. Furthermore,
$
[T]=z.[A].y^{-1}
$
for some $z,y\in U$. Hence
\[
[C]=[T].y=z.[A]\in\cOr_T\cap\cOl_A.
\]
Now $[C]=[T].y$ is special by Proposition~\ref{rplacebij}, and
$\nker[C]=\nker[T]$ by Remark~\ref{normalfixedfromright}. Moreover,
$[C]=z.[A]$ is extraspecial by Lemma~\ref{normtempl yconds}. Since
$[T]$ vanishes at all normal main hook intersections and
$\nker[C]=\nker[T]$, the same is true for $[C]$. Thus $[C]\in\cOl_A$
is an extraspecial linear character vanishing at all normal main hook
intersections. Since $[T]$ is the unique normal template in $\cOr_C$,
we have $\cT_A([C])=[T]$. Hence $\cT_A$ is surjective.
\end{proof}

Theorem~\ref{nker suff} suggests that, in order to classify the left
orbits in $\cObi_A$, we first determine the special linear characters
in $\cOl_A$ that vanish at all normal main hook intersections. This
will be carried out in the next section. In a second step, in
Section~7, we establish a criterion for when two such special linear
characters are mapped under $\cT_A$ to the same normal template. We
then determine representatives of the fibres of $\cT_A$, which in
turn yield representatives of the left orbits in $\cObi_A$.

As we shall see later, there is a similar simplification for the
construction of the right orbits in $\cObi_A$. In this case, however,
normal templates are no longer sufficient. Instead, we shall need all
extraspecial linear characters in $\cOl_A$, including those that do not
vanish at the normal main hook intersections. Thus, in either case,
constructing the extraspecial linear characters in $\cOl_A$ is the key to
decomposing $\cObi_A$ into one-sided orbits.

\section{Construction of the Extraspecial Linear Characters in $\cOl_A$}

In this section we construct the extraspecial linear characters in the left orbit $\cOl_A$ of the normal template $[A]$. We start with some notation:

\begin{Defn}\label{column group} For $1\le s<n$, we define the \textbf{column group} $X_{\col(s)}$ to be the subgroup of $U=U_J$ generated by the root subgroups $X_{js}$ with $1\le j<s$ and $(j,s)\in J$. This is an abelian pattern subgroup of $U$. Hence 
\[ 
X_{\col(s)} = (X_{1,s}\times X_{2,s}\times\cdots\times X_{s-1,s})\cap U_J 
\] 
is the direct product of the root subgroups $X_{js}$ with $(j,s)\in J$. We write 
\[ X_\cS=\prod_{j\in\cS}X_{js}\le X_{\col(s)}
 \] 
for subsets $\cS\subseteq\{1,\ldots,s-1\}$ satisfying $\{(j,s)\mid j\in\cS\}\subseteq J$
 \hfill$\square$ \end{Defn}

 Let $1\le a\le n$. The only way to change the row $\fr_a(B)$ of
$[B]\in\V$ is by applying restricted row operations from the left,
that is, by acting on $[B]$ with a sequence of root group elements
$x_{ia}(\alpha)$, where $1\le i<a$, $(i,a)\in J$, and
$\alpha\in\F_q$. Equivalently, we act by elements of the column group
$X_{\col(a)}$ (see Illustration~\ref{nmupic} below). Since
$X_{\col(a)}$ is abelian, the order of the root group elements in such a
sequence is irrelevant.  
   
\begin{Setup}\label{Gammadef}
Observe that every nonnormal column of $[A]$ contains at most one nonzero entry, which is then necessarily a main condition, since $[A]$ is a normal template by assumption. Fix $(a,b)\in\p=\p[A]$. Until stated otherwise, we shall work with this fixed main condition and therefore omit from the notation introduced below indices indicating dependence on $(a,b)$. Later, when we combine the constructions associated with different main conditions, we shall index the objects as needed.

We define the following subsets of $\{1,\ldots,n\}$: 
\begin{alignat*}{3}
&\cR\, &\, = \, &\{e\,|\,\text{there exists } (e,f)\in\p \text{ with }1\leq e<a<f<b\}\\
&\hC&\,= \,  &\{s\,|\,a<s<b, \text{ and there exists } e\in\cR: A_{es}=\alpha\neq 0\}\\
&\cC&\,  = \,& \{s \,|\,s\in\hC,\, \fc_s(A) \text{ is normal}\}\\
&\cC^\n&\, = \, &\{s\,|\,s\in\cC \text{ there exists }1\leq t<a: (t,s)\in\p\}
\end{alignat*}

We define 
\[ 
\hG=A_{\cR\times\hC} \qquad\text{and}\qquad \G=\hG_{\cR\times\cC}=A_{\cR\times\cC}. 
\] 
Thus $\hG$ is obtained from $[A]$ by deleting all rows at or below row $a$ and all columns not strictly between columns $a$ and $b$. From the remaining matrix we further delete all zero rows and zero columns, as well as all rows containing a main condition strictly to the right of column $b$. The matrix $\G$ is obtained from $\hG$ by deleting all nonnormal columns. We illustrate these sets below:

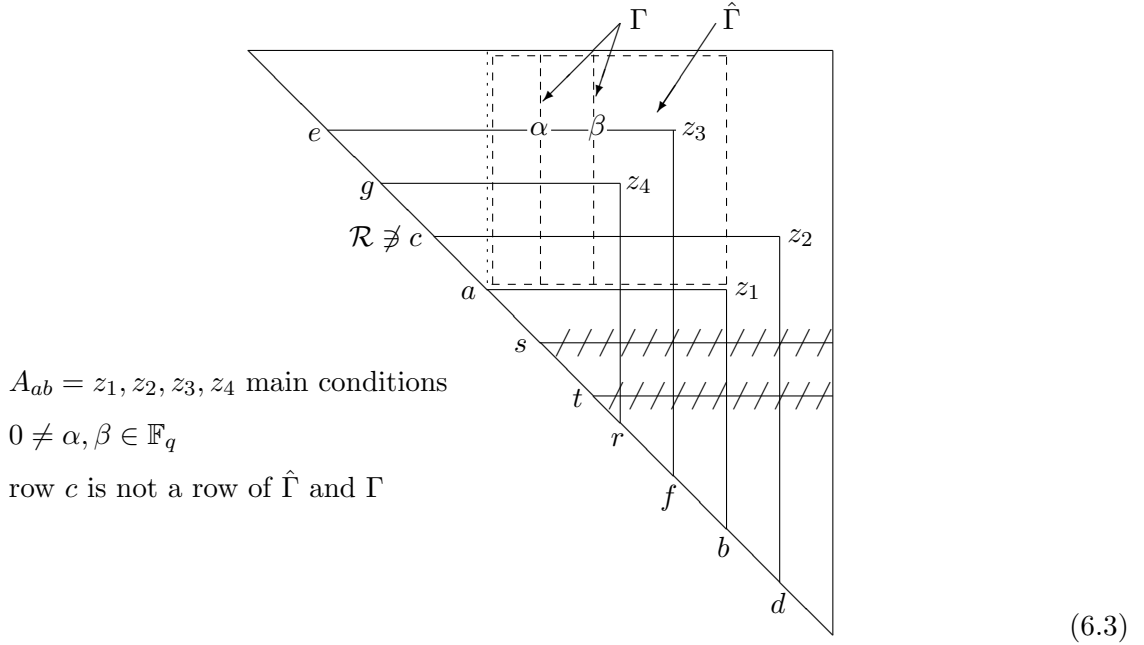
\begin{equation}\label{nmupic}
\begin{picture}(230,235)

\put(10,220){\line(1,0){220}}
\put(10,220){\line(1,-1){220}}
\put(230,220){\line(0,-1){220}}

\put(100,130){\line(1,0){90}}
\put(190,130){\line(0,-1){90}}
\put(90,126.5){\makebox{$a$}}
\put(186.5,29.5){\makebox{$b$}}
\put(193,128){\makebox{$z_1$}}

\multiput(190,132)(0,5.9){15}{\line(0,1){3}}
\multiput(187,132)(-6.1,0){15}{\line(1,0){3}}
\multiput(102,132)(0,5.9){15}{\line(0,1){3}}
\multiput(187,218)(-6,0){15}{\line(1,0){3}}
\dottedline[$\line(0,1){1}$]{4}(100,133)(100,219)

\put(80,150){\line(1,0){130}}
\put(210,150){\line(0,-1){130}}
\put(48,146.5){\makebox{$\cR\not\owns c$}}
\put(206.5,8.5){\makebox{$d$}}
\put(213,148){\makebox{$z_2$}}

\put(40,190){\line(1,0){75}}
\put(124,190){\line(1,0){14}}
\put(145,190){\line(1,0){26}}
\put(170,190){\line(0,-1){130}}
\put(32.5,186){\makebox{$e$}}
\put(165,49){\makebox{$f$}}
\put(173,188){\makebox{$z_3$}}

\put(60,170){\line(1,0){90}}
\put(150,170){\line(0,-1){90}}
\put(52,166){\makebox{$g$}}
\put(146.5,71){\makebox{$r$}}
\put(152,168){\makebox{$z_4$}}

\put(120,110){\line(1,0){110}}
\multiput(126,105)(8.2,0){13}{\line(1,2){5}}
\put(140,90){\line(1,0){90}}
\multiput(146,85)(7.8,0){11}{\line(1,2){5}}
\put(110,106.5){\makebox{$s$}}
\put(132,86){\makebox{$t$}}

\multiput(120,132)(0,5.7){10}{\line(0,1){3}}
\multiput(120,195)(0,5.1){5}{\line(0,1){3}}

\multiput(140,132)(0,5.7){10}{\line(0,1){3}}
\multiput(140,195)(0,5.1){5}{\line(0,1){3}}
\put(116,188){\makebox{$\alpha$}}
\put(138,188){\makebox{$\beta$}}




\put(-80,92){\makebox{$A_{ab}=z_1, z_2,z_3,z_4$ main conditions}}
\put(-80,72){\makebox{$0\neq \alpha,\beta\in \F_q$}}
\put(-80,52){\makebox{row $c$ is not a row of $\hat \Gamma$ and $\Gamma$}}

\put(186,230){\vector(-2,-3){22}}

\put(185,228){\makebox{ $\hat\Gamma$}}

\put(150,230){\vector(-1,-1){29}}
\put(150,230){\vector(-1,-3){9}}

\put(150,228){\makebox{ $\Gamma$}}

\end{picture}
\end{equation}

Observe that $e,g\in\cR$, and parts of rows $e$ and $g$ are rows of $\hG$, 
since the main conditions $z_3,z_4$ lie to the north of row $a$ and strictly 
between columns $a$ and $b$. Columns $r$ and $f$ belong in part to $\G$, and 
$r,f\in\hC$, since $e,g\in\cR$ and the main conditions $A_{ef}=z_3$ and 
$A_{gr}=z_4$ are nonzero entries in these rows.

Rows $s$ and $t$ are not contained in $J$, and hence columns $s$ and $t$ are 
normal. They contain the normal supplementary conditions $\alpha$ and $\beta$, 
respectively, in row $e$, and hence $s,t\in\cC$. Neither column contains a main 
condition to the north of row $a$, and hence $s,t\notin\cC^\n$.

Also $c\notin\cR$, since row $c$ contains the main condition $z_2$ to the right 
of column $b$. If $0\neq\gamma\in\F_q$, then
\[
[B]=x_{ca}(\gamma).[A]
\]
satisfies
\[
B_{ad}=-\gamma z_2\neq0,
\]
and hence $[B]$ is no longer special. However, by 
Proposition~\ref{mainpreserved}, every special $[C]\in\cOl_A$ can be obtained 
from $[A]$ by a sequence of restricted row operations such that every 
intermediate linear character is special. Therefore, row $c$ is not used for 
restricted row operations on row $a$, and accordingly $c\notin\cR$.
\hfill$\square$
\end{Setup}

\begin{Remark}\label{nullmatrix hG}Keep the notation introduced in \ref{Gammadef} above. Then: 
\begin{enumerate} \item[(1)] The rows of $\hG$ are linearly independent, since each row contains a main condition and these main conditions lie in pairwise different columns. Thus $\rank(\hG)=|\cR|$. 
\item[(2)] The matrix $\G$ does not necessarily have full rank $|\cR|$. However, the columns 
$
\{\fc_s(A)\mid s\in\cC^\n\}
$ 
are linearly independent, since each of them contains a main condition and these main conditions lie in pairwise different rows and columns. 
\item[(3)] Columns $s,t$ of $\G$ in Illustration~\ref{nmupic} above intersect the hook arm $h_{ab}^a$ of $[A]$ at $(a,s)$ and $(a,t)$, respectively. Similarly, each main condition $(g,r)\in\p$ with $1\leq g<a<r<b$ gives rise to a main hook intersection 
$ 
(a,r)=h_{ab}^a\cap h_{gr}^l 
$ 
on row $a$ between columns $a$ and $b$. 
\item[(4)] It may happen that $\hG$, and hence $\G$ as well, are empty matrices, for instance when $(a,b)\in\p$ is the highest main condition. In this case, the entries on the hook arm $h_{ab}^a$ cannot be manipulated and therefore remain unchanged throughout the constructions that follow.
\hfill$\square$\end{enumerate}\end{Remark}

Let $i\in\cR$ and $\alpha\in\F_q$. From
Illustrations~\ref{permaction} and~\ref{nmupic}, together with
part~(4) of \ref{colvects}, it follows immediately that
$[B]=x_{ia}(\alpha).[A]$ is obtained by adding $-\alpha$ times the
extension by zeros of the row $\fr_i(\hG)$ of $\hG$ to row $a$ of
$[A]$. More generally, elementary linear algebra yields:

\begin{Lemma}\label{key lin alg} Let $\alpha_i\in\F_q$ for $i\in\cR$ and $x=\prod_{i\in\cR}x_{ia}(-\alpha_i)\in X_\cR\leq X_{\col(a)}$. Then $[B]=x.[A]$ coincides with $[A]$ except in row $a$, which is given by $\fr_a(B)=\fr_a(A)+\sum_{i\in\cR}\alpha_i\fr_i(\hG)$. 
\hfill$\square$ \end{Lemma} 
For $\y=(\alpha_i)_{i\in\cR}\in\F_q^{\cR}$, observe that 
\[ 
\y\hG=\sum_{i\in\cR}\alpha_i\fr_i(\hG)\in\F_q^{\hC}. 
\] 
Thus Lemma~\ref{key lin alg} implies:

\begin{Cor}\label{DefofLinearmap}
With the notation introduced in \ref{Gammadef}, define the linear maps
\[
\hL:\F_q^{\cR}\longrightarrow\F_q^{\hC}:
\y\longmapsto\y\hG\quad\text{and}\quad
L:\F_q^{\cR}\longrightarrow\F_q^{\cC}:
\y\longmapsto\y\G,
\]
where $\y\in\F_q^{\cR}$ is viewed as a row vector.
Moreover, define
\begin{equation}\label{singlecolumnchange}
x_\y=\prod_{i\in\cR}x_{ia}(-\y_i)\in X_\cR\leq U,
\text{ and }
\alpha x_\y=x_{\alpha\y}
\text{\, for }\alpha\in\F_q.
\end{equation}
Then, for $\y,\tilde\y\in\F_q^{\cR}$, 
$
 x_\y x_{\tilde\y}=x_{\y+\tilde\y}. 
$ 
 Thus, $X_\cR$ is an $\F_q$-vector space, and the map $\y\mapsto x_\y$ is an isomorphism from $\F_q^{\cR}$ onto $X_\cR$. Moreover, if $[B]=x_\y.[A]$, then $[B]$ coincides with $[A]$ except in row $a$, which is given by 
 $ 
 \fr_a(B)=\fr_a(A)+\hL(\y). 
$ 
 \hfill$\square$ \end{Cor}

\begin{Remark}\label{hL and L} Note that the rows of $\hG$ span $\im\hL$, and hence $\hL$ is injective by part (1) of Remark~\ref{nullmatrix hG}. However, $L$ is not injective in general. Observe also that $L=\pi\circ\hL$, where $\pi:\F_q^{\hC}\to\F_q^{\cC}$ is the natural projection mapping $\v\in\F_q^{\hC}$ to $\v_{\cC}\in\F_q^{\cC}$. Thus, for $i\in\cC$, we have $\hL(\y)_i=L(\y)_i$. In particular, if $0\neq\y\in\ker(L)$, then $0\neq\z=\hL(\y)$, but $\z_i=0$ for all $i\in\cC$. Therefore $\hL(\ker(L)) = \{\z\in\im\hL\mid \z_i=0\text{ for all }i\in\cC\}$. \hfill$\square$ \end{Remark}

We now need a fact from elementary linear algebra. For the moment, let $\hC$ and $\cR$ be arbitrary finite ordered sets with $|\cR|\leq|\hC|$, and let $\hG\in M_{\cR\times\hC}(q)$ be an arbitrary matrix of full rank $|\cR|$. Define the linear map $\hL:\F_q^\cR\to\F_q^{\hC}$ by $\hL(\v)=\v\hG$ for $\v\in\F_q^\cR$. Then $\hL$ is injective, and the rows $\{\fr_i(\hG)\mid i\in\cR\}$ of $\hG$ form a basis of $\im(\hL)\leq\F_q^{\hC}$.

\begin{Lemma}\label{linalg}
Let $\cZ\subseteq\hC$ be such that
$\{\fc_j(\hG)\mid j\in\cZ\}$ is a basis of the column space of $\hG$.
For each $i\in\hC\setminus\cZ$, write
\[
\fc_i(\hG)
=
\sum_{j\in\cZ}\gamma_{ij}\fc_j(\hG),
\]
where the coefficients $\gamma_{ij}\in\F_q$ are uniquely determined.

Then, for every choice of $\beta_j\in\F_q$, $j\in\cZ$, there exists
a unique $\y\in\F_q^\cR$ such that
$\z=\hL(\y)\in\F_q^{\hC}$ satisfies
\[
\z_i=
\begin{cases}
\beta_i, & \text{if } i\in\cZ,\\
\sum_{j\in\cZ}\gamma_{ij}\beta_j,
& \text{if } i\in\hC\setminus\cZ.
\end{cases}
\]
Moreover, every element of $\im(\hL)$ is obtained in this way.
\end{Lemma}

\begin{proof}
Since the columns $\{\fc_j(\hG)\mid j\in\cZ\}$ form a basis of the
column space of $\hG$ and
$\rank(\hG)=|\cR|=|\cZ|$, the submatrix
$
\Lambda=\hG_{\cR\times\cZ}
$
is square and invertible.

For a given choice of $\beta_j\in\F_q$, $j\in\cZ$, define
$\tilde\z\in\F_q^{\cZ}$ by $\tilde\z_j=\beta_j$ for $j\in\cZ$.
Then
\[
\y=\tilde\z\Lambda^{-1}\in\F_q^{\cR}
\]
is the unique vector satisfying $\y\Lambda=\tilde\z$.

Set $\z=\hL(\y)=\y\hG$. For $s\in\hC$, we have
$\z_s=\y\fc_s(\hG)$. In particular, if $j\in\cZ$, then
\[
\z_j
=\y\fc_j(\hG)
=\y\fc_j(\Lambda)
=\tilde\z_j
=\beta_j.
\]
For $i\in\hC\setminus\cZ$, we obtain
\[
\z_i = \y\fc_i(\hG) = \sum_{j\in\cZ}\gamma_{ij}\y\fc_j(\hG) =  \sum_{j\in\cZ}\gamma_{ij}\beta_j.
\]
Thus $\z$ has the required entries.

Conversely, if $\z\in\im(\hL)$, then $\z=\hL(\y)$ for some
$\y\in\F_q^\cR$, and its entries $\z_j$, $j\in\cZ$, determine
$\y$ uniquely, since $\Lambda$ is invertible. Hence every element of
$\im(\hL)$ is obtained uniquely in this way.
\end{proof}

We return to the situation of \ref{Gammadef}. 

\begin{Defn}\label{row bases}
Keep the notation introduced in \ref{Gammadef}. We now define certain
subsets of the columns of $\hG$, corresponding subsets of $\hC$, and
the intersections of these columns with row $a$. Recall that the
columns of $\G$ indexed by $\cC^\n$ are linearly independent.

\begin{enumerate}
\item[(1)] We define
$
\cB^\n=\{\fc_i(\G)\,|\,i\in\cC^\n\}$ and
$\cH^\n=\{(a,t)\,|\,t\in\cC^\n\}$.
Thus $\cH^\n$ is the set of normal main hook intersections on row $a$
of $[A]$.
\item[(2)] Choose a basis $\tilde\cB$ of the column space of $\G$
containing $\cB^\n$, and define
$
\cB^\fri=\tilde\cB\setminus\cB^\n.
$
Then set
\[
\cC^\fri=\{i\in\cC\,|\,\fc_i(\G)\in\cB^\fri\}
\quad\text{and}\quad
\cN^\fri=\{(a,i)\in J\,|\,\fc_i(\G)\in\cB^\fri\}.
\]

\item[(3)] We denote by
$\bC=\cC\setminus(\cC^\fri\cup\cC^\n)$
the set of remaining column labels of $\G$. We define
\[
\bB=\{\fc_i(\G)\mid i\in\bC\}
\quad\text{and}\quad
\bN=\{(a,t)\mid t\in\bC\}.
\]
Thus
\[
\bN\cup\cN^\fri
=\{(a,t)\mid t\in\bC\cup\cC^\fri\}
\]
is the set of normal supplementary conditions on row $a$ whose columns do not vanish above row~$a$. Note that none of these columns can contain
a main condition above row $a$, since then the corresponding position of row $a$ would be a main hook intersection rather than a supplementary condition.

\item[(4)] Now choose a basis $\hB\supseteq\tilde\cB$ of the column
space of $\hG$, and define
$\hB^\fa=\hB\setminus\tilde\cB$.
Thus $\hB^\fa$ consists of the column vectors in $\hB$ indexed by
nonnormal columns. We define
\[
\hC^\fa=\{i\mid\fc_i(\hG)\in\hB^\fa\}
\quad\text{and}\quad
\hH^\fa=\{(a,t)\in J\mid t\in\hC^\fa\}.
\]
Thus the basis $\hB$ of the column space of $\hG$ is the disjoint union
$
\hB=\cB^\n\cup\cB^\fri\cup\hB^\fa.
$

\item[(5)] Finally, we set
$\hC^\fri=\hC\setminus(\cC\cup\hC^\fa)$.
Then we define
\[
\hB^\fri=\{\fc_t(\hG)\mid t\in\hC^\fri\}
\quad\text{and}\quad
\hH^\fri=\{(a,t)\mid t\in\hC^\fri\}.
\]
Thus $\hH^\fri\cup\hH^\fa$ is the disjoint union of the sets of
nonnormal main hook intersections on row~$a$ indexed by
$\hC^\fri$ and $\hC^\fa$, respectively.\hfill$\square$
\end{enumerate}
\end{Defn}

\begin{Remark}\label{row bas prop}
Observe that
\[
\hC
=
\cC^\n\cup\cC^\fri\cup\bC\cup\hC^\fa\cup\hC^\fri
\]
is a set partition of $\hC$. Hence
$\hB^\fh=\cB^\n\cup\hB^\fa\cup\hB^\fri$
is a disjoint union as well.

The set $\hB^\fh$ consists of all nonnormal column vectors of $\hG$
together with those normal column vectors that contain a main condition,
and hence precisely of all column vectors of $\hG$ that contain a main
condition. These are linearly independent, since the
main conditions lie in pairwise different rows and columns. Moreover,
each row of $\hG$ contains precisely one main condition. Hence
$\hB^\fh$ has cardinality $|\cR|$ and is a basis of the column space
of $\hG$.

Since $\hB$ is also a basis of the column space of $\hG$, we have
\begin{equation}\label{comp bases}
|\hB|
=
|\cB^\n|+|\cB^\fri|+|\hB^\fa|
=
|\hB^\fh|
=
|\cB^\n|+|\hB^\fa|+|\hB^\fri|,
\end{equation}
and hence
$
|\cB^\fri|=|\hB^\fri|.
$
The corresponding positions on $h_{ab}^a$ are precisely the main hook
intersections
$
\hH^\fh
=
\cH^\n\cup\hH^\fa\cup\hH^\fri
$
on row $a$.
\hfill$\square$
\end{Remark}

Note that, for every $s\in\bC\cup\hC^\fri$, we can uniquely write
\begin{equation}\label{col deprel}
\fc_s(\hG)
=
\sum_{\fc_t(\hG)\in\hB}
\gamma_{st}\fc_t(\hG)
\end{equation}
with $\gamma_{st}\in\F_q$. If $s\in\bC$, then
$\fc_s(\hG)=\fc_s(\G)$ lies in the span of $\tilde\cB$, and hence
$\gamma_{st}=0$ for all $\fc_t(\hG)\in\hB^\fa$.
With this notation, Lemma~\ref{linalg} and the injectivity of $\hL$
(see Remark~\ref{hL and L}) imply:

\begin{Prop}\label{extspecial local}
For every choice of $\beta_s\in\F_q$ for
$(a,s)\in\cH^\n\cup\cN^\fri\cup\hH^\fa$, there exists a unique
extraspecial linear character $[B]\in\cOl_A$ which coincides with
$[A]$ at all positions except possibly those in row $a$. The entries
of row $a$ are given by
\begin{equation}\label{entries row a}
B_{as}
=
\begin{cases}
\beta_s,
&
\text{if }(a,s)\in\cH^\n\cup\cN^\fri\cup\hH^\fa,\\
A_{as}
+
\sum_{t\in\cC^\n\cup\cC^\fri\cup\hC^\fa}
\gamma_{st}(\beta_t-A_{at}),
&
\text{if }(a,s)\in\hH^\fri,\\
A_{as}
+
\sum_{t\in\cC^\n\cup\cC^\fri}
\gamma_{st}(\beta_t-A_{at}),
&
\text{if }(a,s)\in\bN,
\\
A_{as},
&
\text{otherwise}.
\end{cases}
\end{equation}
Moreover, every extraspecial linear character which coincides with
$[A]$ except possibly in row $a$ is obtained by a unique choice of
the entries $\beta_s\in\F_q$.
\end{Prop}

\begin{proof}
Recall first that, apart from the main condition, the only nonzero
entries in row $a$ of $[A]$ are normal supplementary conditions, since
$[A]$ is a normal template. Thus $A_{as}\neq 0$ implies
$(a,s)\in\cN^\fri\cup\bN$. Set
\begin{equation}\label{deltas}
\delta_s=
\begin{cases}
\beta_s,
& \text{if }(a,s)\in\cH^\n\cup\hH^\fa,\\
\beta_s-A_{as},
& \text{if }(a,s)\in\cN^\fri.
\end{cases}
\end{equation}
Then
$\delta_s=\beta_s-A_{as}$ for all
$(a,s)\in\cH^\n\cup\cN^\fri\cup\hH^\fa$.

Recall the injective linear map
$\hL:\F_q^{\cR}\to\F_q^{\hC}$, $\y\mapsto\y\hG$, defined in
\ref{DefofLinearmap}. Applying Lemma~\ref{linalg}, we obtain a unique
$\y\in\F_q^\cR$ such that $\hL(\y)=\v\in\F_q^{\hC}$ satisfies
$v_s=\delta_s$ for all
$s\in\cC^\n\cup\cC^\fri\cup\hC^\fa$.
For $s\in\bC\cup\hC^\fri$, we have
\begin{equation}\label{vau sub s}
\v_s
=
\sum_{t\in\cC^\n\cup\cC^\fri\cup\hC^\fa}
\gamma_{st}\delta_t
=
\sum_{t\in\cC^\n\cup\cC^\fri\cup\hC^\fa}
\gamma_{st}(\beta_t-A_{at}).
\end{equation}

By Corollary~\ref{DefofLinearmap},
$[B]=x_\y.[A]$ satisfies
$\fr_a(B)=\fr_a(A)+\y\hG=\fr_a(A)+\v$.
Hence
$B_{as}=A_{as}+\delta_s=\beta_s$ for
$(a,s)\in\cH^\n\cup\cN^\fri\cup\hH^\fa$.
For $(a,s)\in\bN\cup\hH^\fri$, equation~\ref{vau sub s} gives
\[
B_{as}
=
A_{as}
+
\sum_{t\in\cC^\n\cup\cC^\fri\cup\hC^\fa}
\gamma_{st}(\beta_t-A_{at}),
\]
while all remaining entries of $[A]$ are unchanged. Moreover, \ref{col deprel} implies that $\gamma_{st}=0$ whenever $s\in\bC$ and
$t\in\hC^\fa$. This proves
Equation~\ref{entries row a}.

By construction, $[B]\in\cOl_A$ is extraspecial. Conversely, every
extraspecial linear character in $\cOl_A$ which coincides with $[A]$
outside row $a$ arises from a unique choice of the entries $\beta_s$,
by Lemma~\ref{linalg} and the injectivity of $\hL$. This proves the
last assertion.
\end{proof}

Here is a special case, which will be needed later and
follows immediately from Equation~\ref{entries row a}:

\begin{Cor}\label{choice rl}
Choose $\beta_s=A_{as}$ for
$(a,s)\in\cH^\n\cup\cN^\fri$. Then $B_{as}=A_{as}$ for
$(a,s)\in\bN$ as well. Hence $[B]$ and $[A]$ coincide at all normal
positions in row $a$.
\hfill$\square$
\end{Cor}

So far, we have concentrated on a single main hook $h_{ab}$ of $[A]$.
We now combine the constructions for all main hooks of $[A]$. We first
introduce some notation:

\begin{Notation}\label{setup 2}
Let
$\p[A]=\p=\{(i_1,j_1),\ldots,(i_k,j_k\})$. We order $\p$ from bottom
to top so that $i_1>i_2>\cdots>i_k$, in contrast to the ordering \ref{p order 1} of
$\p$ used in Proposition~\ref{rplacebij}.

For $1\leq\mu\leq k$, we now attach the subscript $\mu$ to the objects
associated with the main hook $h_{i_\mu j_\mu}$ that were defined in
\ref{Gammadef}, \ref{DefofLinearmap}, and \ref{row bases}. 
\hfill$\square$
\end{Notation}

Note that all rows $\fr_i(A)$ with $i<i_k$ are zero rows, and hence
$\hG_k$ is the empty matrix (cf. part (4) of Remark ~\ref{nullmatrix hG}). Consequently,
$\fr_{i_k}(B)=\fr_{i_k}(A)$ for every extraspecial linear character
$[B]\in\cOl_A$.

We can now apply the method of Proposition~\ref{extspecial local} to
construct all extraspecial linear characters in $\cOl_A$, working
through the main hook arms from bottom to top. Set $[B_0]=[A]$, and
suppose that, for $1\leq\mu\leq k-1$, we have constructed an
extraspecial linear character $[B_{\mu-1}]\in\cOl_A$ by successively
changing entries at main hook intersections and normal supplementary
conditions on the hook arms
$h_1^a,\ldots,h_{\mu-1}^a$ using Proposition~\ref{extspecial local}.

To manipulate the entries on $h_\mu^a$, observe that no entry above
row $i_{\mu-1}$ has been changed in the construction of
$[B_{\mu-1}]$. Hence
\[
(B_{\mu-1})_{\cR_\mu\times\hC_\mu}
=
A_{\cR_\mu\times\hC_\mu}
=
\hG_\mu,
\]
and, in particular,
$\fr_{i_\mu}(B_{\mu-1})=\fr_{i_\mu}(A)$.

Working from bottom to top through the main hook arms ensures that all
submatrices, main hook legs, and linear maps involved are determined
by $[A]$ from the outset. Consequently, it makes no difference whether
we regard the procedure as successively manipulating the main hook
arms $h_1^a,h_2^a,\ldots,h_k^a$ in passing from
$[B_0]=[A]$ to $[B_1],\ldots,[B_k]$, or as manipulating the hook arms
$h_\mu^a$, $1\leq\mu\leq k$, of $[A]$ independently of one another.

In Setup \ref{Gammadef}  and Definition~\ref{row bases}, we defined certain sets of columns of the matrices
$\hG_\mu$ and corresponding positions on the main hook arms $h_\mu^a$,
for $\mu=1,\ldots,k$, omitting the indices associated with the
individual main conditions. Henceforth, the corresponding symbols
without indices will denote the unions of these sets over all main
conditions.

\begin{Defn}\label{Hook inters}
Recall that
$\cH_\mu^\n=\{(i_\mu,s)\in J\mid s\in\cC_\mu^\n\}$
denotes the set of normal main hook intersections on the hook arm
$h_\mu^a$, as defined in part (1) of Definition~\ref{row bases}.
We set
\[
\cH^\n
=
\cH_1^\n\cup\cH_2^\n\cup\cdots\cup\cH_k^\n.
\]
Thus $\cH^\n$ is the set of all normal main hook intersections of
$\p$. Similarly, we set
\[
\cN^\fri=\bigcup_{\mu=1}^k\cN_\mu^\fri
\qquad\text{and}\qquad
\bN=\bigcup_{\mu=1}^k\bN_\mu.
\]
The sets $\hH^\fa$, $\hH^\fri$, $\cC^\n$, and the other sets introduced
in Setup~\ref{Gammadef} and Definition~\ref{row bases} are defined
analogously as the disjoint unions of their local parts associated
with the individual main hook arms.

Moreover, we abbreviate
$
\cH=\cH^\n\cup\cN^\fri\cup\hH^\fa.
$
Note that $\cN^\fri\cup\bN$ is the set of normal supplementary
conditions.
\hfill$\square$
\end{Defn}

Proposition~\ref{extspecial local} and Corollary~\ref{choice rl} now
imply immediately the following:

\begin{Theorem}\label{extspecial global}
For every choice of
$\bbeta=(\beta_{rs})\in\F_q^\cH$, there exists a unique extraspecial
linear character $[A(\bbeta)]=[B]\in\cOl_A$ satisfying:
\begin{enumerate}
\item[(i)] $B_{rs}=\beta_{rs}$ for all $(r,s)\in\cH$.

\item[(ii)] The entries $B_{i_\mu t}$ at positions
$(i_\mu,t)\in\hH^\fri_\mu\cup\bN_\mu$, for
$\mu=1,\ldots,k$, are determined by Equation~\ref{entries row a}
applied to the main hook $h_{i_\mu j_\mu}$.

\item[(iii)] $B_{rs}=A_{rs}$ for all remaining positions
$(r,s)\in J$.
\end{enumerate}
Every extraspecial linear character in $\cOl_A$ is obtained in this
way for a unique choice of $\bbeta\in\F_q^\cH$.

Moreover, if $\beta_{rs}=A_{rs}$ for all
$(r,s)\in\cH^\n\cup\cN^\fri$, then $B_{rs}=A_{rs}$ for all
$(r,s)\in\bN$ as well.
\hfill$\square$
\end{Theorem}

\begin{Remark}
Instead of prescribing the values $\beta_{rs}\in\F_q$ for
$(r,s)\in\cH=\cH^\n\cup\cN^\fri\cup\hH^\fa$ in
Theorem~\ref{extspecial global}, we may use the alternative basis
$\hB^\fh$ of Remark~\ref{row bas prop} and prescribe the entries at
the positions in
$\hH^\fh=\cH^\n\cup\hH^\fri\cup\hH^\fa$, the set of all main hook
intersections.
\hfill$\square$
\end{Remark}

Recall from \ref{Q sets} that, for $W\subseteq\V$, $\fE(W)$ denotes
the set of extraspecial linear characters in $W$.

\begin{Notation}\label{part choices}
For $\cW\subseteq\cH$, set
\[
\fE^\cW
=
\{[A(\bbeta)]\mid
\bbeta=(\beta_{rs})\in\F_q^\cH,\ 
\beta_{rs}=A_{rs}\text{ for all }(r,s)\in\cH\setminus\cW\}.
\]
Thus $\fE^\cW$ consists of those extraspecial linear characters in
$\cOl_A$ for which only the entries indexed by $\cW$ are allowed to
vary. Note that $\fE^\cH$ is the set of all extraspecial linear characters in $\cOl_A$ by \ref{extspecial global}.

For the particular choices
$\cW=\cH^\n,\cN^\fri$, and $\hH^\fa$, we abbreviate
$
\fE^{\cH^\n}=\fE^\n,\,
\fE^{\cN^\fri}=\fE^\fri,\,
\fE^{\hH^\fa}=\fE^\fa,
$
respectively.
\hfill$\square$
\end{Notation}

\section{Decomposing $\cObi_A$ into Left Orbits}

Recall from Theorem~\ref{nker suff} that, in order to decompose $\cObi_A$ into
left orbits, it suffices to determine the special linear characters in
$\cOl_A$ that vanish at all normal main hook intersections. In the
previous section, we constructed all extraspecial linear characters in
$\cOl_A$. We now make suitable choices for the entries at the main hook
intersections and normal supplementary conditions in
Theorem~\ref{extspecial global}, and then at the $y$-conditions using
Proposition~\ref{normalenough}, to obtain a complete classification of
the left orbits in $\cObi_A$.

Note that different extraspecial linear characters $[B]\in\cOl_A$
which vanish at normal main hook intersections may map to the same
normal template under the surjective map $\cT_A$ of
Theorem~\ref{nker suff}. Thus, in order to use these characters to
classify the left orbits in $\cObi_A$, we have to determine the
fibres of $\cT_A$. For this purpose, we first inspect
$
\cOrl_A=\cOr_A\cap\cOl_A
$
in view of \ref{HomSpace}. 

Recall from Proposition~\ref{rplacebij}
that the right orbit $\cOr_A$ consists entirely of special linear
characters. Moreover, by \ref{normalfixedfromright}, all
$[B]\in\cOr_A$ have the same normal kernel,
$
\nker[B]=\nker[A].
$
 We define

\begin{equation}\label{auto vectors}
\cA(A)=\cA
=
\left\{
\bbeta=(\beta_{rs})\in\F_q^\cH
\,\middle|\,
\beta_{rs}=A_{rs}
\text{ for all }(r,s)\in\cH^\n\cup\cN^\fri
\right\}.
\end{equation}

Thus, $\bbeta\in\F_q^\cH$ belongs to $\cA$ if and only if
$\beta_{rs}=0$ for $(r,s)\in\cH^\n$,
$\beta_{rs}=A_{rs}$ for $(r,s)\in\cN^\fri$, and
$\beta_{rs}\in\F_q$ is arbitrary for $(r,s)\in\hH^\fa$.

\begin{Theorem}\label{main1}
Let $\bbeta=(\beta_{rs})\in\cA$. Then there exists a unique
extraspecial linear character
$[A(\bbeta)]=[B]\in\cOrl_A$ satisfying:
\begin{enumerate}
\item[(i)] $B_{ij}\neq A_{ij}$ implies
$(i,j)\in\hH^\fa\cup\hH^\fri$;

\item[(ii)] $B_{rs}=\beta_{rs}$ for
$(r,s)\in\hH^\fa$;

\item[(iii)] the entries $B_{i t}$ at
$(i,t)\in\hH^\fri\cup\bN$ are determined by
Equation~\ref{entries row a}; in particular,
$B_{it}=A_{it}$ for
$(i,t)\in\bN$, hence $\nker[B]=\nker[A]$.
\end{enumerate}
Moreover, every $[B]\in\cOrl_A$ is obtained in this way and 
\[
\cOrl_A
=
\{[A(\bbeta)]\mid\bbeta\in\cA\}
=
\fE^\fa.
\]
In particular,
$|\cOrl_A|=q^\alpha$, where
$\alpha=|\hH^\fa|$.
\end{Theorem}
 
 \begin{proof}
By Theorem~\ref{extspecial global}, for every $\bbeta\in\cA$ there
exists a unique extraspecial linear character
$[B]=[A(\bbeta)]\in\cOl_A$ with the prescribed entries in
$\hH^\fa$, while the remaining entries on the main hook arms are
determined by Equation~\ref{entries row a}.
Applying Corollary~\ref{choice rl} to all main hooks of $[A]$ shows
that $B_{rs}=A_{rs}$ at all normal positions $(r,s)\in J$. Hence
$\nker[B]=\nker[A]$. Proposition~\ref{rplacebij} now yields
$[B]\in\cOr_A$, and therefore $[B]\in\cOrl_A$.

If $[C]\in\cOrl_A$, then in particular $[C]\in\cOl_A$ is extraspecial
with $\nker[C]=\nker[A]$. Hence the entries of $[C]$ on
$\cH^\n\cup\cN^\fri$ agree with those of $[A]$. Applying
Theorem~\ref{extspecial global}, we see immediately that $[C]$ is one
of the linear characters $[B]\in\cOrl_A$ constructed above and the
the theorem follows by \ref{HomSpace}.
\end{proof}

Applying \ref{HomSpace}, we have:

\begin{Cor}\label{endoring} The elements of $\fE^\fa$ determine bases of $\End_{\C U}(\C\cOr_A)$ and $\End_{\C U}(\C\cOl_A)$.
\hfill$\square$\end{Cor}

As an immediate consequence of \ref{HomSpace} and Theorem~\ref{main1},
we obtain necessary and sufficient conditions for the left orbit modules
$\C\cOl_B$, with $[B]\in\V$, to be irreducible. Since
$[B]\in\cObi_A$ for some normal template $[A]\in\V$ by
Corollary~\ref{normaltempl}, and $\C\cOl_B\cong\C\cOl_A$, we may
assume that $[B]=[A]$.

\begin{Cor}\label{main2} Let $[A]\in\V$ be a normal template and suppose $\p[A] = \{(i_1,j_1),\ldots,(i_k,j_k)\}$ such that $i_1>\ldots>i_k$. Then the following statements are equivalent:
\begin{enumerate}
\item[(1)] For all $1\leq\mu\leq k$ the column space of  $\hG_\mu$ is spanned by its normal columns.
\item[(2)] $\hH^\fa_\mu = \emptyset$ for all $1\leq\mu\leq k$. 
\item[(3)] $\cOrl_A = \{[A]\}$.
\item[(4)] $\dim_\C(\End_{\C U}(\C\cOr_A)) = \dim_\C(\End_{\C U}(\C\cOl_A)) = 1$.
\item[(5)] $\C\cOr_A$ and $\C\cOl_A$ are irreducible $\C U_J$-modules.\hfill$\square$
\end{enumerate}
\end{Cor}

\begin{Remark}\label{order hook inters}  Recall that we denoted the number of one-sided orbits in $\cObi_A$ by $\kappa$ in \ref{gen orbit facts}. We now set 
$\alpha = |\hH^\fa| = |\hH^\fa_1|+\ldots+|\hH^\fa_k|$  and  similarly $\iota = |\hH^\fri|$. Note that then $\iota = |\cN^\fri|$ as well by \ref{row bas prop} (extended to all main hooks by \ref{Hook inters}). Moreover we denote the number of $y$-conditions of $\p$ by  $\upsilon = |\bfy(\p)|$. Then we have in view of \ref{OrOl bij}
\begin{equation}\label{place count}
q^{|\lPl(\p)|} =|\cOl_A|=|\cOr_A|= q^{|\rPl(\p)|} = q^{\alpha+\iota+\upsilon} \text{ and } \kappa = \frac{|\cOr_A|}{|\cOrl_A|} = \frac{q^{\alpha+\iota+\upsilon}}{q^\alpha} = q^{\iota+\upsilon}  
\end{equation}
observing that $\rPl(\p) = \hH^\fa\cup\hH^\fri\cup\bfy$ (disjoint union) by  \ref{rplacebij} and applying \ref{main1}.
\hfill$\square$\end{Remark}

Recall from \ref{part choices} the definition of $\fE^\fri$.
Theorem~\ref{extspecial global} now implies:

\begin{Prop}\label{praemain3}
For every $\bbeta\in\F_q^\cH$ such that $\beta_{rs}\in\F_q$ is
chosen arbitrarily for $(r,s)\in\cN^\fri$ and
$\beta_{rs}=A_{rs}=0$ for
$(r,s)\in\cH^\n\cup\hH^\fa$, there exists a unique extraspecial
linear character $[A(\bbeta)]=[B]\in\cOl_A$ which vanishes at all
normal main hook intersections and satisfies:
\begin{enumerate}
\item[(i)] $B_{rs}=\beta_{rs}$ for
$(r,s)\in\cH$. In particular,
$B_{rs}=\beta_{rs}=0$ for
$(r,s)\in\cH^\n\cup\hH^\fa$.

\item[(ii)] The entries $B_{rs}$ at positions
$(r,s)\in\bN\cup\hH^\fri$ are uniquely determined by
Equation~\ref{entries row a}.

\item[(iii)] $B_{rs}=A_{rs}$ for all other positions $(r,s)\in J$.
\end{enumerate}
The set $\fE^\fri$ is precisely the set of extraspecial linear characters
$[B]\in\cOl_A$ obtained in this way. It contains $q^\iota$ distinct
extraspecial linear characters with pairwise different normal kernels.
\hfill$\square$
\end{Prop}

We now set
$\cH^\bfy=\cH\cup\bfy
=\cH^\n\cup\cN^\fri\cup\hH^\fa\cup\bfy(\p)$.
Applying Theorem~\ref{extspecial global} and
Proposition~\ref{normalenough}, we make suitable choices for the
scalars $\beta_{rs}\in\F_q$, for $(r,s)\in\cH^\bfy$, and formulate
our main result on the left orbits contained in $\cObi_A$.
Essentially, this result was already proved in [\cite{DG3}, 6.22],
after applying the mirror map \ref{mirror}. Here, however, we include
a much simplified proof. We retain the notation and definitions
introduced in \ref{row bases} and recall the notation introduced in
\ref{Q sets}.

Consider the right orbits $\cOr_B$ for $[B]\in\fE^\fri$. Since $[B]$
is special, each $\cOr_B$ is a special right orbit. Moreover, all its
members vanish at the normal main hook intersections, since $[B]$
does and normal kernels are constant on right orbits by
\ref{normalfixedfromright}. Recall from \ref{extraspecial exists} that every special right orbit
$\cOr_B$ contains a unique extraspecial linear character $[E^B]$
which, in addition, vanishes at all nonnormal main hook intersections
$\hH^\fri\cup\hH^\fa$. Thus $[E^B]$ vanishes at all main hook
intersections and hence is a normal template.

Recall the notation introduced in \ref{Q sets}. In particular,
$\fT_A$ denotes the set of templates in $\cObi_A$. Define
\begin{equation}\label{Iso scalars}
\cI
=
\{\bbeta=(\beta_{rs})\in\F_q^{\cH^\bfy}
\mid \beta_{rs}=0\text{ for all }
(r,s)\in\cH^\n\cup\hH^\fa\}.
\end{equation}
We now obtain the following decomposition of $\cObi_A$ into left
orbits:

\begin{Theorem}\label{main3}
Let $[A]\in\V$ be a normal template with main condition set
$\p=\p[A]$. Then, for every $\bbeta=(\beta_{rs})\in\cI$, there exists
a unique template $[T]=[T(\bbeta)]\in\fT_A$ such that
$T_{rs}=\beta_{rs}$ for all $(r,s)\in\cH^\bfy$.

More precisely, $\cObi_A$ decomposes into the disjoint union of left
orbits
\[
\cObi_A
=
\bigcup_{[T]\in\fT_A}\cOl_T
=
\bigcup_{\bbeta\in\cI}\cOl_{T(\bbeta)}.
\]
Thus $\cI$ parametrizes the left orbits contained in $\cObi_A$, and
$
|\cI|=q^{\iota+\upsilon}=\kappa.
$
\end{Theorem}

\begin{proof}
By Proposition~\ref{praemain3}, the set $\fE^\fri$ consists of $q^\iota$
distinct extraspecial linear characters with pairwise distinct normal
kernels, all of which vanish at the normal main hook intersections.
Each right orbit $\cOr_B$ with
$[B]\in\fE^\fri$ contains by \ref{extraspecial exists} a unique normal template $[E^B]$. Hence
\begin{equation}\label{normal all}
\{[E^B]\mid [B]\in\fE^\fri\}
=
\{[T(\bbeta)]\mid
\bbeta\in\cI,\ 
\beta_{rs}=0\text{ for all }(r,s)\in\bfy\}
\subseteq\fTn_A.
\end{equation}

By Proposition~\ref{normalenough}, all templates in $\cObi_A$ are
obtained by filling the $y$-conditions of the normal templates in
$\fTn_A$ with arbitrary entries from $\F_q$. Hence we obtain
$q^\iota q^\upsilon=q^{\iota+\upsilon}=\kappa$ distinct templates in
$\cObi_A$. By \ref{gen orbit facts}, this is precisely the number of
left orbits contained in $\cObi_A$. Therefore these templates form a
complete set of representatives for the left orbits, we have equality in \ref{normal all}, and the asserted
decomposition of $\cObi_A$ into left orbits follows.
\end{proof}

\section{Decomposing $\cObi_A$ into Right Orbits}

In \ref{normalenough} we proved that inserting arbitrary entries at the
$y$-conditions of a normal template in $\cObi_A$ produces all templates
and hence all left orbits in $\cObi_A$. The key ingredient is the
existence of a unique template in each left orbit, as stated in
\ref{unique one}. As mentioned earlier, for right orbits we lack a
corresponding distinguished object. Nevertheless, using a different
approach, we shall show that a similar construction works for right
orbits as well. We first need the following:

\begin{Defn}\label{w positions}
Define $\bfw=\bfw(\p)$ to be the set of positions in $J$ on the main
hook legs that are not main hook intersections. The elements of $\bfw$
are called {\bf $w$-conditions}.
\hfill$\square$
\end{Defn}

Note that normal main hook intersections are $\ell$-places as well and
have to be treated similarly to $w$-conditions. The reason is that the
entries at normal main hook intersections can be changed by restricted
row operations, but not by restricted column operations. Thus distinct
entries at these positions keep right orbits apart and hence behave,
with respect to the right action of $U$ on $\V$, like normal
supplementary conditions. We have:

\begin{Lemma}\label{w and y}
Let $\p=\p[A]$. Then
$|\bfw(\p)|=|\bfy(\p)|-|\cH^\n|$.
\end{Lemma}

\begin{proof}
By part iii) of \ref{template}, the $y$-conditions are the positions
in $J$ on the reduced main hook arms with respect to $\p$ that are not
main hook intersections. Thus, by \ref{rplaces},
$
\bfy
=
\rPl\setminus(\hH^\fri\cup\hH^\fa).
$
Similarly, \ref{w positions} and \ref{lplaces} imply
$
\bfw
=
\lPl\setminus(\hH^\fri\cup\hH^\fa\cup\cH^\n).
$
Since $|\rPl|=|\lPl|$ by \ref{OrOl bij}, it follows that
$|\bfw|=|\bfy|-|\cH^\n|$, proving the lemma.
\end{proof}

As we shall see, inserting different entries at the $w$-conditions of
extraspecial linear characters in $\cOr_A$ produces linear characters
in $\cOr_A$ that generate isomorphic but pairwise distinct left orbits.
As in our analysis of the left orbits in $\cObi_A$, extraspecial linear
characters in $\cOl_A$ play a central role in our investigation of the
right orbits. Here, however, the extraspecial linear characters
$[B]\in\cOl_A$ need not vanish at the normal main hook intersections
in $\cH^\n$. Allowing nonzero entries at these positions compensates
precisely for the difference between the cardinalities of $\bfy$ and
$\bfw$.

The following lemma is an immediate consequence of \ref{HomSpace},
Theorem~\ref{main1}, and the fact that isomorphic modules have
isomorphic endomorphism rings:

\begin{Lemma}\label{Auts are iso} Let $[B]\in\cObi_A$. Then $|\cOrl_B|=|\cOrl_A| = q^\alpha$. 
\hfill$\square$\end{Lemma}

 Recall the definition of $\cA$ in \ref{auto vectors} and that
$
\cObi_A=\bigcup_{[B]\in\cOl_A}\cOr_B
$
by \ref{biorbitdec}. Thus it suffices to consider the right orbits
$\cOr_B$ with $[B]\in\cOl_A$.

\begin{Lemma}\label{left right intersection I}
Let $[B]=b.[A]\in\cOl_A$ with $b\in U$. Then, for every $u\in U$,
\[
[A].u\in\cOrl_A
\quad\text{if and only if}\quad
[B].u\in\cOrl_B.
\]
Consequently,
\[
b.(\cOrl_A)
=
\{b.[A(\bbeta)]\mid\bbeta\in\cA\}
=
\cOrl_B.
\]
\end{Lemma}

\begin{proof}
Let $u\in U$ with $[A].u\in\cOrl_A$. Then
$[A].u=c.[A]$ for some $c\in U$. Hence
\[
[B].u
=
b.[A].u
=
(bc).[A]
=
(bcb^{-1}).b.[A]
=
(bcb^{-1}).[B]\in\cOl_B.
\]
Since $[B].u\in\cOr_B$, it follows that
$[B].u\in\cOrl_B$.

Conversely, let $u\in U$ be such that
$[B].u=d.[B]\in\cOrl_B$ for some $d\in U$. Then
\[
b.[A].u
=
[B].u
=
d.[B]
=
(db).[A],
\]
and hence
\[
[A].u=(b^{-1}db).[A]\in\cOl_A.
\]
Since $[A].u\in\cOr_A$, we obtain
$[A].u\in\cOrl_A$. This proves the first assertion.

Now let $[C]\in\cOrl_B$. Then
$[C]=[B].u=b.[A].u$ for some $u\in U$. By the first assertion,
$[A].u\in\cOrl_A$, and therefore
$[C]\in b.(\cOrl_A)$. Thus
\[
\cOrl_B\subseteq b.(\cOrl_A).
\]
Conversely, if $[C]\in b.(\cOrl_A)\subseteq\cOr_A$, then
$[C]=b.[A].u=[B].u$ for some $u\in U$ with
$[A].u\in\cOrl_A$. Again by the first assertion,
$[B].u\in\cOrl_B$, and hence
$
b.(\cOrl_A)\subseteq\cOrl_B.
$
Therefore
$
b.(\cOrl_A)=\cOrl_B.
$
Finally, Theorem~\ref{main1} gives
$
\cOrl_A
=
\{[A(\bbeta)]\mid\bbeta\in\cA\},
$
and consequently
\[
\cOrl_B
=
b.(\cOrl_A)
=
\{b.[A(\bbeta)]\mid\bbeta\in\cA\},
\]
as desired.
\end{proof}

Note that the column group $X_{\col(i)}$ for a normal column $i$ acts
trivially on $\V$ by the left dot action. Therefore, we set $u_i=1$
without further notice whenever row $i$ is not contained in $J$. Note in particular, $u_n=1$.
Moreover, \ref{startfacts} implies that every $u\in U=U_J$ can be
written uniquely as
$
u=u_2u_3\cdots u_{n-1},
$
where, for $2\leq i\leq n-1$, $u_i$ is a uniquely determined element
of the abelian column subgroup $X_{\col(i)}$ (see
\ref{column group}), and $u_i=1$ if row $i$ is not contained in $J$.
Each $u_i$ can in turn be written uniquely as
\[
u_i=\prod_{d=1}^{i-1}x_{di}(\zeta_{di})\in X_{\col(i)},
\]
where $\zeta_{di}\in\F_q$ is uniquely determined for $(d,i)\in J$
and $\zeta_{di}=0$ otherwise.

We now proceed similarly to Lemma~\ref{key lin alg}. The left dot
action of $u_i$ on any $[C]\in\V$ can be illustrated as follows,
where main conditions are indicated by $\circ$:

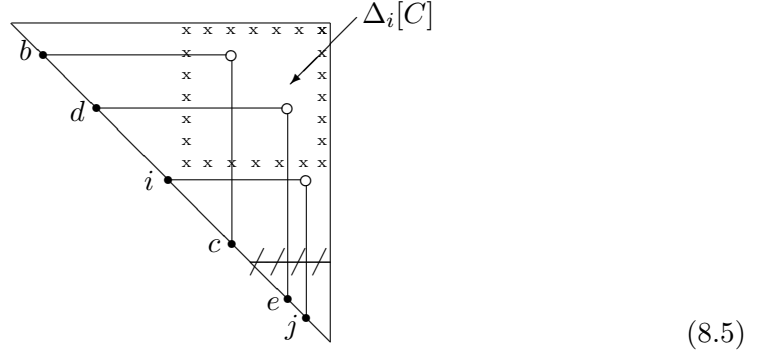
\begin{equation}\label{Deltamatrix}
\begin{picture}(120,120)
\put(130,120){\line(0,-1){120}}
\put(130,120){\line(-1,0){120}}
\put(10,120){\line(1,-1){120}}

\put(90,105){$\circ$}
\put(91,108){\line(-1,0){69}}
\put(93,106){\line(0,-1){69}}
\put(22,108){\circle*{3}}
\put(13,106){\makebox{$b$}}

\dottedline[\tiny x]{9}(84,67)(120,67)
\dottedline[\tiny x]{9}(76,67)(76,117)
\dottedline[\tiny x]{9}(84,117)(127,117)
\dottedline[\tiny x]{9}(127,117)(127,67)

\put(142,120){\makebox{$\Delta_i[C]$}}
\put(140,122){\vector(-1,-1){25}}

\put(93,37){\circle*{3}} 
\put(84,34){\makebox{$c$}}

\put(69,61){\circle*{3}} 
\put(60,57){\makebox{$i$}}

\dottedline{.2}(100,30)(130,30)
\multiput(100,25)(7.8,0){4}{\line(1,2){5}}

\put(118,58){$\circ$}
\put(121,59){\line(0,-1){49}}
\put(68,61){\line(1,0){51}}

\put(121,9){\circle*{3}} 
\put(113,3){$j$}

\put(114,16){\circle*{3}} 
\put(106,12){$e$}

\put(111,85){$\circ$}
\put(112,88){\line(-1,0){69}}
\put(114,86){\line(0,-1){69}}
\put(42,88){\circle*{3}}
\put(32,83){\makebox{$d$}}

\end{picture}
\end{equation}

Define, for $2\leq i\leq n-1$,
$
\Delta_i[C]=\Delta_i
=
C_{\{1,\ldots,i-1\}\times\{i+1,\ldots,n\}},
$
the matrix obtained from $C$ by deleting rows $i,\ldots,n$ and
columns $1,\ldots,i$. 

The left
dot action of $u_i\in X_{\col(i)}$ on $[C]$ amounts to subtracting from row $i$ the
linear combination of the rows of $\Delta_i$ with coefficients
$\zeta_{di}\in\F_q$, for $d=1,\ldots,i-1$. Thus,  we have
\begin{equation}\label{use Delta}
\fr_i(u_i.[C])
=
\fr_i(C)-\sum_{d=1}^{i-1}\zeta_{di}\fr_d(\Delta_i),
\end{equation}
where the rows of $\Delta_i$ are understood to be extended by zeros.

Suppose now that $[C]$ is extraspecial and that $(b,c),(i,j)\in\p[C]$. 
Consider the nonnormal main hook
intersection $(i,c)\in J$. If $C_{dc}\neq0$ for some $1\leq d<i$, then
$b<d$, and $(d,c)\in h_{bc}^l$ is a nonnormal main hook intersection
strictly between the main condition $(b,c)$ and the main hook
intersection $(i,c)$.

Let $\hH(i,c)$ denote the set of main hook intersections on  $h_{bc}^l$ strictly to the north of $(i,c)$. Then
\begin{equation}\label{use Delta2}
D_{ic}
=
C_{ic}-\sum_{d=1}^{i-1}\zeta_{di}C_{dc}=
C_{ic}
-
\left(
\zeta_{bi}C_{bc}
+
\sum_{(d,c)\in\hH(i,c)}
\zeta_{di}C_{dc}
\right).
\end{equation}

Observe that 
\begin{equation}\label{use Delta3}
\Delta_t[C] = \Delta_t[D] \text{ for all } 2\leq t\leq i,
\end{equation}
since $[C]$ and $[D]$ differ only on row $i$. 
\hfill$\square$

Lemma~\ref{left right intersection I} provides the tools to prove the following proposition:

\begin{Prop}\label{premain 4}
Let $[B]\in\cOl_A$. Then, for any choice of $\xi_{ij}\in\F_q$ for
$(i,j)\in\hH^\fa$, there exists a unique $[C]\in\cOrl_B$ such that
$\nker[C]=\nker[B]$ and $C_{ij}=\xi_{ij}$ for all
$(i,j)\in\hH^\fa$. Moreover, every linear character in $\cOrl_B$
arises in this way.
\end{Prop}

\begin{proof} 
Let $[B]=u.[A]$ for some $u\in U$. Then  
$
\cOrl_B
=
\{u.[A(\bbeta)]\mid\bbeta\in\cA\}
$
by
Lemma~\ref{left right intersection I}.

Fix $\bbeta=(\beta_{rs})\in\cA$ and set
$
[B(\bbeta)]=u.[A(\bbeta)].
$
Write
$
u=u_2u_3\cdots u_{n-1},
$
where $u_i\in X_{\col(i)}$ is uniquely determined. Set
$[A^n]=[A(\bbeta)]$ and, for $i=n-1,n-2,\ldots,2$, define
\[
[A^i]
=
u_i.[A^{i+1}]
=
u_i.u_{i+1}\cdots u_{n-1}.[A(\bbeta)]\in\cOl_A.
\]
Thus
$
[B(\bbeta)]
=
u.[A(\bbeta)]
=
[A^2].
$

By construction, the rows are changed from bottom to top, while the
submatrix $\Delta_i=\Delta_i[A(\bbeta)]$ lies entirely to the north
of row $i$. Hence the changes made in rows $i,i+1,\ldots,n-1$ do not
affect $\Delta_i$. Thus, comparing with \ref{use Delta3}, we have
\begin{equation}\label{use Delta4}
\Delta_i[A^a]
=
\Delta_i[A(\bbeta)]
=
\Delta_i
\qquad\text{for all }a\geq i.
\end{equation}

Acting by $u_i$ on $[A^{i+1}]$ changes only row $i$ of
$[A^{i+1}]$, and this row is not changed again by the subsequent
actions of $u_{i-1},\ldots,u_2$. Consequently,
\begin{equation}\label{use Delta5}
\fr_i(A^i)
=
\fr_i(A^{i-1})
=
\ldots
=
\fr_i(A^2)
=
\fr_i(u.[A(\bbeta)])
=
\fr_i(B(\bbeta)).
\end{equation}
On the other hand, the preceding actions of
$u_{n-1},\ldots,u_{i+1}$ do not change row $i$, and hence
\begin{equation}\label{use Delta6}
\fr_i(A^{i+1})
=
\fr_i(A^{i+2})
=
\ldots
=
\fr_i(A^{n-1})
=
\fr_i(A(\bbeta)).
\end{equation}
 
Now let $u_i = x_{1i}(\zeta_{1i})x_{2i}(\zeta_{2i})\cdots x_{i-1,i}(\zeta_{i-1,i})$ with $\zeta_{di}\in\F_q$ for $d=1,\ldots,i-1$ setting $\zeta_{di}=0$, if row $d$ is not contained in $J$. Then
\begin{equation}\label{use Delta7}
\fr_i(B(\bbeta)) = \fr_i(A^i)
 = \fr_i(A^{i+1}) - \sum_{d=1}^{i-1}\zeta_{di}\fr_d(\Delta_i[A(\bbeta)])
=\fr_i(A(\bbeta)) - \sum_{d=1}^{i-1}\zeta_{di}\fr_d(\Delta_i[A(\bbeta)])
\end{equation}
applying \ref{use Delta5}, \ref{use Delta6}, \ref{use Delta} and \ref{use Delta4}.

We are interested only in the entries $B(\bbeta)_{rs}$ at nonnormal
main hook intersections $(r,s)\in\hH^\fa$ of $[B(\bbeta)]$. Thus let
$
(i,c)=h_{ij}^a\cap h_{bc}^l
$
for some 
$
b<i<c<j<n,
$
as in Illustration~\ref{Deltamatrix}.

Thus $A(\bbeta)_{dc}\neq0$ implies either $d=b$ or
\[
(d,c)=h_{de}^a\cap h_{bc}^l
\in\hH^\fa\cup\hH^\fri
\]
is a nonnormal main hook intersection. In the latter case, $b<d<c<e\leq n$. Moreover, if $(d,c)\in\hH^\fri$, we define $\beta_{dc}=A(\bbeta)_{dc}$ and recall that Proposition~\ref{extspecial local} expresses $\beta_{dc}$ as a linear combination 
\begin{equation}\label{coeff dc}
\beta_{dc}
=
A_{dc}
+
\sum_{(d,s)\in\hH^\fa}
\gamma_{ds}(\beta_{ds}-A_{ds}),
\end{equation}
observing that $\bbeta\in\cA$ implies
$\beta_{ds}=A_{ds}$ for
$(d,s)\in\cH^\n\cup\cN^\fri$.
In particular, the coefficients $\gamma_{ds}$ in
\ref{coeff dc} depend only on the normal template $[A]$ we started
with.

Now \ref{use Delta7} and \ref{use Delta2} imply $\xi_{ic} = B(\bbeta)_{ic} = \beta_{ic} - \cK$
where
\begin{equation}\label{use Delta9}
\cK=\cK(i,c)= \sum_{d=b}^{i-1}\zeta_{di}A(\bbeta)_{dc}=
\zeta_{bi}A(\bbeta)_{bc}+
\sum_{(d,c)\in\hH^\fa\cup\hH^\fri} 
\zeta_{di}\beta_{dc}
\end{equation}

Observe that $\cK$ depends on the fixed $\zeta_{di}$, the
$\beta_{dc}$ for $d<i$, and the coefficients $\gamma_{ds}$ in
\ref{coeff dc}, all of which are independent of $\beta_{ic}$.
Consequently, $\xi_{ic}$ runs through $\F_q$ as $\beta_{ic}$ runs
through $\F_q$.

Applying this procedure to all main hook intersections
$(i,c)\in\hH^\fa$, proceeding from bottom to top along the main hook
arms, we obtain a vector
\begin{equation}\label{xxi}
\xxi(\bbeta) = \xxi = (\xi_{ic})\in\F_q^{\hH^\fa}
\end{equation}
such that $B(\bbeta)_{ic}=\xi_{ic}$ for all $(i,c)\in\hH^\fa$. Moreover, the correspondence 
\[
\F_q^{\hH^\fa}\longrightarrow\F_q^{\hH^\fa}:\bbeta\longmapsto\xxi
\]
is a bijection. Hence, for every prescribed
$\xxi=(\xi_{ic})\in\F_q^{\hH^\fa}$, we have constructed a unique
linear character $[B(\bbeta)]\in\cOrl_B$ satisfying $B(\bbeta)_{ic} = \xi_{ic}$ for all $(i,c)\in\hH^\fa$. In particular, these linear characters 
are pairwise distinct. Thus we obtain
$
q^{|\hH^\fa|}=q^\alpha
$  many, 
pairwise distinct linear characters in $\cOrl_B$. By
Lemma~\ref{Auts are iso}, these are all the elements of $\cOrl_B$. Finally, since $[B(\bbeta)]\in\cOr_B$, we have $\nker[B(\bbeta)] = \nker[B]$ by Remark~\ref{normalfixedfromright}.
\end{proof}

\begin{Remark}\label{nker not preserved}
Inspecting Equation~\ref{use Delta7}, we see that the action of $u$
on $[A(\bbeta)]$ in Proposition~\ref{premain 4} may also change
entries in normal columns. We have
\[
\nker[A(\bbeta)]=\nker[A]\quad\text{and}\quad \nker[B(\bbeta)]=\nker[B]
\]
for $[B(\bbeta)]=u.[A(\bbeta)]$ as above. However,
$\nker[B]$ may differ from $\nker[A]$.
\hfill$\square$
\end{Remark}

Proposition~\ref{premain 4} now immediately implies the following
main result:

\begin{Theorem}\label{main 4}
Let $[A]\in\V$ be a normal template. Define
\[
\fI_A
=
\{[B]\in\cOl_A\mid
B_{ts}=0\text{ for all }(t,s)\in\hH^\fa\}.
\]
Then
\[
\{\cOr_B\mid [B]\in\fI_A\}
\]
is the complete set of pairwise disjoint right orbits in $\cObi_A$.
\end{Theorem}

\begin{proof}
Let $[B],[C]\in\fI_A$. Then
$\cOl_B=\cOl_A=\cOl_C$. Suppose that $\cOr_B=\cOr_C$. Since
$[C]\in\cOl_B$ and $[C]\in\cOr_C=\cOr_B$, we have
$[C]\in\cOrl_B$. By \ref{normalfixedfromright},
$
\nker[B]=\nker[C].
$
Moreover, $B$ and $C$ both vanish at all positions
$(t,s)\in\hH^\fa$. Hence the uniqueness assertion of
Proposition~\ref{premain 4} implies that $[B]=[C]$. Thus distinct
elements of $\fI_A$ generate distinct right orbits.

It remains to show that every right orbit in $\cObi_A$ is obtained
in this way. By \ref{biorbitdec}, every such right orbit is of the
form $\cOr_D$ for some $[D]\in\cOl_A$. Applying
Proposition~\ref{premain 4} to $[D]$, with
$\xi_{ij}=0$ for all $(i,j)\in\hH^\fa$, yields a unique
$[B]\in\cOrl_D$ such that
$
B_{ij}=0
\text{ for all }(i,j)\in\hH^\fa.
$
Thus $[B]\in\fI_A$, and since $[B]\in\cOr_D$, we have
$\cOr_B=\cOr_D$. Hence every right orbit in $\cObi_A$ has a
representative in $\fI_A$. Together with the uniqueness established
above, this proves the theorem.
\end{proof}

We can now count the right orbits contained in $\cObi_A$ directly.
Set
\[
\omega=|\bfw|,\qquad
\sigma=|\cH^\n|,\qquad
\iota=|\hH^\fri|=|\cN^\fri|,
\qquad
\alpha=|\hH^\fa|.
\]
By \ref{place count} and \ref{w and y}, we have
$
|\cOl_A|
=
q^{\omega+\sigma+\iota+\alpha}.
$
Since $|\cOrl_A|=q^\alpha$ by Theorem~\ref{main1}, it follows that
\[
|\fI_A|
=
q^{\omega+\sigma+\iota}
=
q^{\upsilon+\iota}
=
\kappa,
\]
where, as before, $\upsilon=|\bfy|$ and $\kappa$ denotes the number
of one-sided orbits contained in $\cObi_A$; here we have used
\ref{w and y} and part~(3) of \ref{gen orbit facts}.

Moreover, as in the proof of Theorem~\ref{main1}, the uniqueness
assertion of Theorem~\ref{main 4} yields the following left-hand
analogue of \ref{normalenough}:

\begin{Cor}\label{w cond rev} Let $[B],[C]\in\cOl_A$ with $B_{ts}=C_{ts}$ for all $(t,s)\in\hH^\fa$ and $\nker[B]=\nker[C]$. Suppose that $B_{ij}\neq C_{ij}$ for some $(i,j)\in\bfw$. Then $\cOr_B\cap\cOr_C = \emptyset$.
\hfill$\square$\end{Cor}

Finally, we collect our main results into the following summary:

\begin{Summary}\label{summ} Let $[A]\in\V$ be a normal template. Keep the notation introduced in the preceding sections. In particular, recall from \ref{Hook inters} the notation for the various types of main hook intersections  and normal supplementary conditions.

Combining Propositions \ref{placebij} and \ref{rplacebij} with Theorem \ref{extspecial global} we conclude that for each choice of 
$\bbeta=(\beta_{ij})\in\F_q^{\cH^\bfy}$ there exists a unique special linear character $[S]=[S(\bbeta)]\in\cObi_A$ with 
$\main[A]=\main[S]$ 
 such that $S_{ij}=\beta_{ij}$ for $(i,j)\in\cH$ and $S_{ij}$ is determined by Equation \ref{entries row a} for $(i,j)\in\hH^\fri\cup\bN$. By specializing 
$\bbeta$, fixing its entries on various parts of $\cH = \cH^\n\cup\cN^\fri\cup\hH^\fa\cup\bfy$ we obtain our main results on the coadjoint orbits of $\V$.
\begin{enumerate}
\item[1)] The main hook intersections $\hH^\fa$ parametrize the intersection 
$\cOrl_A$. Indeed, setting $\beta_{rs} = 0$ for $(r,s)\in\cH^\n\cup\bfy$, $\beta_{rs}=A_{rs}$ for 
$(r,s)\in\cN^\fri$, and choosing
$\beta_{rs}\in\F_q$ arbitrary for $(r,s)\in\hH^\fa$, we obtain
precisely the elements of $\cOrl_A$. Thus $\fE^\fa$ determines bases of 
$\End_{\C U}(\C\cOr_A)$ and $\End_{\C U}(\C\cOl_A)$; compare with Theorem \ref{main1}.
\item[2)] The normal suplementary conditions in $\cN^\fri$ together with the $y$-conditions $\bfy$ parametrize the left orbits contained in $\cObi_A$. Indeed, setting
\[
\beta_{rs}=0\text{  for all } (r,s)\in\cH^\n\cup\hH^\fa\cup\bfy,
\]
and choosing $\beta_{rs}\in\F_q$ arbitrarily for
$(r,s)\in\cN^\fri$, clearing nonzero entries from nonnormal hook intersections, and filling $y$-conditions entreis from $\F_q$ we obtain all templates in $\cObi_A$. These generate the left orbits in $\cObi_A$; compare with \ref{main3}. 
\item[3)]  The right orbits contained in $\cObi_A$ are parametrized by
the linear characters in $\cOl_A$ which vanish at all nonnormal main
hook intersections in $\hH^\fa$ by Theorem \ref{main 4}.
\end{enumerate} 
Moreover, we determine necessary and sufficient conditions for $\C\cO_A$ for being irreducible.
\hfill$\square$\end{Summary}

Using the mirror map \ref{mirror}, one immediately obtains the
corresponding classification of the left $U_{\bar J}$-orbits on
$\V_{\bar J}$ by decomposing the biorbits. Thus the classification
for column closed pattern subgroups of $U_n(q)$ follows as well.
Since this amounts only to translating the results obtained here
via the mirror map, we leave the details to the reader.

\section*{Acknowledgement}This work was supported by the National Natural Science Foundation of China (Grant No. 11601338), and the first author was grateful to the Institute of Algebra and Number Theory at University of Stuttgart  for hosting her research visit, where part of this work was carried out.

\providecommand{\bysame}{\leavevmode ---\ }
\providecommand{\og}{``} \providecommand{\fg}{''}
\providecommand{\smfandname}{and}
\providecommand{\smfedsname}{\'eds.}
\providecommand{\smfedname}{\'ed.}
\providecommand{\smfmastersthesisname}{M\'emoire}
\providecommand{\smfphdthesisname}{Th\`ese}

\end{document}